\documentclass[oneside,english,A4,reqno, 10pt]{amsart}
\usepackage[T1]{fontenc}
\usepackage{orcidlink}
\usepackage[utf8]{inputenc}
\usepackage{xcolor}
\usepackage{amstext}
\usepackage{amsthm}
\usepackage{amssymb}
\usepackage{enumitem}
\usepackage[a4paper]{geometry}
\makeatletter
\numberwithin{equation}{section}
\numberwithin{figure}{section}
\theoremstyle{plain}
\newtheorem{thm}{\protect\theoremname}
\theoremstyle{plain}
\newtheorem{prop}[thm]{\protect\propositionname}
\theoremstyle{definition}
\newtheorem{defn}[thm]{\protect\definitionname}
\theoremstyle{remark}
\newtheorem{rem}[thm]{\protect\remarkname}
\theoremstyle{plain}
\newtheorem{lem}[thm]{\protect\lemmaname}

\allowdisplaybreaks
\usepackage{hyperref}
\usepackage{cleveref}
\usepackage{microtype}
\usepackage{tikz}
\usepackage{pgfplots}
\usepgfplotslibrary{fillbetween}

\makeatother

\usepackage{babel}
\providecommand{\definitionname}{Definition}
\providecommand{\lemmaname}{Lemma}
\providecommand{\propositionname}{Proposition}
\providecommand{\remarkname}{Remark}
\providecommand{\theoremname}{Theorem}

\begin{document}
\title[Quadratic discrepancy estimates for probability measures on the Heisenberg
group]{Quadratic discrepancy estimates for probability\\ measures on the Heisenberg
group}
\author{Luca Brandolini\orcidlink{0000-0002-9670-9051}}
\address{Dipartimento di Ingegneria Gestionale, dell\textquoteright Informazione
e della Produzione, Universit\`a degli Studi di Bergamo, Viale Marconi
5, 24044 Dalmine, Italy}
\email{luca.brandolini@unibg.it, alessandro.monguzzi@unibg.it}
\author{Alessandro Monguzzi \orcidlink{0000-0003-3233-5000}}
\author{Matteo Monti\orcidlink{0000-0001-6848-5938}}
\address{Dipartimento di Matematica e Applicazioni, Universit\`a di Milano--Bicocca,
Via Cozzi 55, 20125 Milano, Italy}
\email{matteo.monti@unimib.it}
\thanks{
All the authors are members of Indam--Gnampa. This work was supported by the European Union – NextGenerationEU, under the National Recovery and
Resilience Plan (NRRP), Mission 4, Component 2, Investment 1.1, funding call PRIN 2022 D.D. 104 published
on 2.2.2022 by the Italian Ministry of University and Research (Ministero dell’Universit\`a e della Ricerca), Project
Title: TIGRECO - TIme-varying signals on Graphs: REal and COmplex methods – CUP F53D23002630001. This project started while M. M. was a postdoctoral researcher on the TIGRECO project at the University of Bergamo. M.M. gratefully acknowledges the research support and stimulating environment provided by the University of Bergamo. 
}
\begin{abstract}
We initiate the study of quadratic discrepancy for finite point sets on the Heisenberg group $\mathbb H^n$
 with respect to upper Ahlfors regular probability measures. For a natural family of test sets given by left translations and dilations of cylindrically defined neighborhoods, we introduce an $L^2$-discrepancy and establish a Roth-type lower bound depending on the homogeneous dimension of $\mathbb H^n$.

This result extends classical discrepancy estimates from the Euclidean and compact settings to a non-commutative, step-two nilpotent Lie group. It should be viewed as a first step toward the development of a discrepancy theory on the Heisenberg group.
\end{abstract}
\maketitle
\section{Introduction}

The study of the distribution of finite point sets and their discrepancy
lies at the intersection of several areas of mathematics, such as
number theory, harmonic analysis, and geometry. It originates from
the fundamental problem of understanding how well a discrete set of
points can approximate a continuous object. In many contexts, one
is given a measure space endowed with a probability measure and seeks
to replace the integration of a function with respect to that measure
by the average of its values over a finite number of sampling points.
The accuracy of such an approximation depends on how uniformly the
points are distributed with respect to the measure, and this notion
is quantitatively captured by the concept of discrepancy. In general
terms, one considers a measure space $X$, a finite sequence of points
$z_{1},z_{2},\ldots,z_{N}\in X$ and a probability measure $\mu$
on $X$. Given a sufficiently rich family $\mathcal{R}$ of measurable
subsets of $X$ the discrepancy of the point set with respect to $R\in\mathcal{R}$
is defined as 
\[
D_{N}\left(R\right)=\sum_{j=1}^{N}\chi_{R}\left(z_{j}\right)-N\mu\left(R\right).
\]
The size of $D_{N}\left(R\right)$ measures the deviation between
the distribution of the points and the reference measure $\mu$, when
tested against the sets in $\mathcal{R}$. Studying how $D_{N}\left(R\right)$
behaves as $N$ grows for different choices of $\mathcal{R}$
provides deep insight into the interplay between discrete and continuous
structures. When the family $\mathcal{R}$ depends
on one or more parameters -- for instance position and dilation of
a given shape -- one can study not only the pointwise behavior of
$D_{N}\left(R\right)$ but also its average with respect to the parameters.
In this setting, it is natural to consider the $L^{2}$-discrepancy,
which measures the mean square deviation of $D_{N}\left(R\right)$
over the parameter space. A landmark result in this direction is the following theorem of Roth
\cite{Roth1954}.
\begin{thm}
There exists $c>0$ such that for any finite sequence $z_{1},z_{2},\ldots,z_{N}$
of points in $[0,1]^{n}$ we have
\[
\bigg(\int_{\left[0,1\right]^{n}}\bigg|\sum_{j=1}^{N}\chi_{I_{x}}\left(z_{j}\right)-N\left|I_{x}\right|\bigg|^{2}dx\bigg)^{1/2}>c\log\left(N\right)^{\frac{n-1}{2}}
\]
where $I_{x}=\left[0,x_{1}\right]\times\cdots\times\left[0,x_{n}\right]$
for every $x=\left(x_{1},\ldots,x_{n}\right)\in\left[0,1\right]^{n}.$
\end{thm}

Since Roth\textquoteright s breakthrough, the study of $L^{2}$-discrepancy
has developed in many directions, involving different families of
subsets, different underlying spaces and different norms. Most of the
contributions have focused on the unit cube, the torus, and the sphere.
See e.g. \cite{BeckChenBook1987}, \cite{BrandoliniGiganteTravaglini2014},
\cite{ChazelleBook2000}, \cite{Panorama2014}, \cite{ChenLectures}, 
\cite{DrmotaTichy1997}, \cite{Matousek1999}, \cite{TravagliniBook2014}
for background and further references on the topic.

More recently, Brandolini, Colzani and Travaglini have extended this
line of research to upper Ahlfors probability measure in $\mathbb{R}^{n}$ \cite{BCT2023} while
Skriganov \cite{Skriganov2017,Skriganov2019} and Brandolini, Gariboldi,
Gigante and Monguzzi \cite{BGG2021,BGG2022} and \cite{BGGM2024} 
have studied the $L^{2}$-discrepancy on compact metric spaces and
on compact homogeneous spaces. Motivated by these developments, it is natural to ask whether analogous
estimates hold in the non-commutative settings of the Heisenberg group
$\mathbb{H}^{n}$, which plays a fundamental role in several areas
of analysis, from partial differential equations to sub-Riemannian
geometry.

In this paper, we initiate the study of $L^{2}$-discrepancy on $\mathbb H^n$. We introduce an appropriate notion of discrepancy and use harmonic analysis to obtain a Roth-type lower bound.
To state our main results we introduce some essential definitions; more details are referred to the next section. 

The Heisenberg group $\mathbb{H}^{n}$ is the space $\mathbb C^n\times\mathbb R$
endowed with the group operation
\[
\left(z,t\right)\circ\left(w,s\right)=\left(z+w,t+s+\frac{1}{2}\mathtt{Im}(z\cdot\overline{w})\right),
\]
where $z,w\in\mathbb C^n$ and $t,s\in\mathbb R$. 
A natural metric on $\mathbb H^n$ is induced by
the Korányi norm
\[
\left\Vert \left(z,t\right)\right\Vert =\left(\left|z\right|^{4}+t^{2}\right)^{1/4},
\]
namely
\[
\mathrm{dist}\left((z,t),(w,s)\right)=\left\Vert (z,t)^{-1}\circ (w,s)\right\Vert.
\]
We denote by $\mathfrak{B}_r=\mathfrak{B}_r(z,t)$ the ball of radius $r$  centered at $(z,t)$ defined by the Korányi norm and we refer to it as a  Heisenberg ball. 
For every $\rho>0$ consider the anisotropic dilations
\[
D_{\rho}\left(z,t\right)=\left(\rho z,\rho^{2}t\right).
\]
It is not hard to check that such dilations are  group automorphisms, namely
\[
D_{\rho}\left[\left(z,t\right)\circ\left(w,s\right)\right]=D_{\rho}\left(z,t\right)\circ D_{\rho}\left(w,s\right).
\]
Let $\left|A\right|$ denotes the Lebesgue measure of a measurable
set $A\subseteq\mathbb{H}^{n}$. It is well known that this measure
is both left- and right-invariant with respect to the group action. Moreover,
if $D_{\rho}A=\left\{ x\in\mathbb{H}^{n}:D_{\rho^{-1}}x\in A\right\}$,
then
\[
\left|D_{\rho}A\right|=\rho^{Q}\left|A\right|
\]
where $Q=2n+2$ is the homogeneous dimension of $\mathbb{H}^{n}$. The Korányi norm $\|\cdot\|$  is homogeneous with respect to the dilations $D_\rho$ in the sense that 
\[
\left\Vert D_{\rho}\left(z,t\right)\right\Vert =\rho\left\Vert \left(z,t\right)\right\Vert .
\]
%
We refer to \cite{BLU} for background on $\mathbb{H}^n$ and homogeneous groups. See also \cite{BB2023}, which contains a concise chapter on homogeneous groups.

For every $\rho>0$, we also let 
\begin{align*}
    B_{\rho}=D_\rho B_1=\left\{ \left(z,t\right)\in\mathbb{H}^{n}:\left|z\right|\leqslant\rho,\left|t\right|\leqslant\rho^{2}\right\}.
\end{align*}
Let $\mu$ be a probability measure on $\mathbb{H}^{n}$ and $\left\{ \left(\zeta_{j},\tau_{j}\right)\right\} _{j=1}^{N}\subseteq\mathbb{H}^{n}$ be  
a finite sequence of points. 
We define the discrepancy
\begin{align*}
D_{N}\left(z,t;\rho\right)= & \sum_{j=1}^{N}\chi_{\left(z,t\right)\circ B_{\rho}}\left(\zeta_{j},\tau_{j}\right)-N\mu\left(\left(z,t\right)\circ B_{\rho}\right)
\end{align*}
where $\left(z,t\right)\circ B_{\rho}=\left\{ \left(w,s\right)\in\mathbb{H}^{n}:\left(z,t\right)^{-1}\circ\left(w,s\right)\in B_{\rho}\right\} $
is the left translation of $B_{\rho}$. 

\medskip
Our main result reads as follows.
\begin{thm}\label{MainResult}
Let $\mu$ be a probability measure on $\mathbb{H}^{n}$ which is
Ahlfors upper regular; namely, there exists $c>0$ such that for every Heisenberg ball $\mathfrak{B}_{r}$ we have $\mu\left(\mathfrak{B}_{r}\right)\leqslant cr^{Q}$.
There exists a constant $C>0$ such that, for every finite sequence
of $N$ points in $\mathbb{H}^{n}$, with $N$ large enough, we have
\[
\int_{0}^{1}\int_{\mathbb{H}^{n}}\left|D_{N}\left(z,t;\rho\right)\right|^{2}dzdt\,d\rho\geqslant CN^{1-\frac{1}{Q-2}}.
\]
\end{thm}
The proof of the theorem combines classical arguments of Roth and Montgomery with the harmonic analysis on the Heisenberg group, in particular with the use of the group Fourier transform defined in terms of the Schrödinger representations of $\mathbb{H}^n$.

Let us clarify why cylindrically symmetric sets arise naturally in our analysis. 
The group Fourier transform maps a function in $L^{1}(\mathbb H^{n})$ to a family of operators on $L^{2}(\mathbb R^{n})$. 
For cylindrically radial functions, these operators have a particularly convenient form: they are diagonal in the Hermite basis. 
This motivates defining the discrepancy with respect to the sets $B_\rho$, which form a simple, explicit family of cylindrically radial sets. 
Moreover, this choice allows us to exploit an explicit description of the Fourier transform of cylindrically radial functions on the Heisenberg group; see Section~\ref{s:transform_radial}.

From a geometric point of view, it would also be natural to consider discrepancy with respect to Heisenberg balls associated with the Korányi metric. The corresponding $L^{2}$-discrepancy problem, however, leads to a substantially different Fourier-analytic setting and is not addressed in this work. In the same circle of ideas, we point out to the reader the papers \cite{GNT, G}.

It is natural to ask whether Theorem \ref{MainResult} is sharp. In Section \ref{sec:upper bound} we apply a general upper bound for the discrepancy of Ahlfors regular measures on metric measure spaces to our setting, which yields \(N^{1-\frac{1}{Q}}\). At present, we do not know which rate is the correct one. The available results in the metric measure space setting are quite general and do not incorporate the geometry of the families of test sets used to evaluate the discrepancy; rather, they only depend on the dimension of their boundaries. Moreover, examples on the torus suggest that when the boundaries of the test sets exhibit singularities, the discrepancy can be smaller than what one would otherwise expect, see \cite{BBM2026} and \cite{BT2022}. This latter observation supports the conjecture that our lower bound is the optimal one.

The rest of the paper is organized as follows. Sections~\ref{s:heisenberg} and \ref{s:transform_radial} collect the basic material on the Heisenberg group and recall the harmonic analysis tools used throughout the paper. In Section \ref{s:main_result} we prove our main result (Theorem \ref{MainResult}); its proof relies  on two main estimates that will be proved in Sections \ref{sec:Estimate I} and \ref{sec:Estimate J}. The preliminary technical ingredients needed to prove such estimates are provided in Section \ref{s:estimates_laguerre}.

In the entire paper we adopt the following notation conventions. Given two quantities $A$ and $B$ that depend on some parameter $x$,
in the following we will write $A\approx B$ meaning that there are
two positive constants $c,C>0$ independent of the relevant values
of the parameter $x$ such that
\[
c\, B \leqslant A \leqslant C\, B.
\]
If one can take $c=C$, we also write $A\simeq B$. Moreover, we write $A\lesssim B$ (resp.\ $A\gtrsim B$) to denote a one-sided inequality
$A\leqslant C\,B$ (resp.\ $A\geqslant c\,B$) for some constant $C>0$ (resp.\ $c>0$).

We also do not attempt to keep track of constants: a positive constant, usually denoted by $c$, may change from line to line. Finally, whenever we are dealing with an integration on $\mathbb R^n$ or $\mathbb H^n$, unless specified, it is understood that the underlying measure is the Lebesgue measure.

We conclude the introduction by thanking Alessio Martini for stimulating discussions and
valuable suggestions concerning the heat kernel on the Heisenberg
group.

\section{The group Fourier transform}\label{s:heisenberg}
In this section we recall the definition of the group Fourier transform, its main properties and we state the Plancherel identity. For further details, we refer the reader to the standard references \cite{F, T_book}.

\subsection{The Schr\"odinger representation of $\mathbb{H}^{n}$ and the group
Fourier transform} Here and in the subsequent subsections we recall some facts about the group Fourier transform on $\mathbb H^n$. For details we refer the reader to the standard references \cite{F}, \cite{Geller}, and \cite{T_book}. 

Identifying $z=x+iy$ we can also describe $\mathbb H^n$ as the space $\mathbb R^{2n+1}$ endowed with the group operation
\[
\left(x_{1},y_{1},t_{1}\right)\circ\left(x_{2},y_{2},t_{2}\right)=\left(x_{1}+x_{2},y_{1}+y_{2},t_{1}+t_{2}+\frac{1}{2}\left(y_{1}x_{2}-x_{1}y_{2}\right)\right).
\]

Let $\lambda\neq0$, and consider the map $\pi_{\lambda}$ from $\mathbb{H}^{n}$
to the space of bounded linear operators on $L^{2}(\mathbb{R}^{n})$,
defined by
\[
\pi_{\lambda}(z,t)\varphi(\xi)=\pi_{\lambda}(x,y,t)\varphi(\xi)=e^{i\lambda\left(x\cdot\xi+\frac{1}{2}x\cdot y+t\right)}\varphi(\xi+y),
\]
for $(z,t)=(x,y,t)\in\mathbb{H}^{n}$ and $\varphi\in L^{2}(\mathbb{R}^{n})$.
This defines a unitary representation of the Heisenberg group $\mathbb{H}^{n}$
on the Hilbert space $L^{2}(\mathbb{R}^{n})$, known as the \emph{Schr\"odinger
representation}. Observe that setting $\pi_{\lambda}(z):=\pi_{\lambda}(z,0)$ we can
also write
\begin{equation}
\pi_{\lambda}\left(z,t\right)\varphi=e^{i\lambda t}\pi_{\lambda}\left(z\right)\varphi.\label{eq: pi lambde}
\end{equation}

The following proposition summarizes the main properties of $\pi_{\lambda}$.
\begin{prop}
\label{prop: proprieta Shrodinger}For every $\lambda\neq0$, the
following hold true.
\end{prop}

\begin{enumerate}[label=(\textit{\roman*})]
\item For every $(z,t)\in\mathbb{H}^{n}$, the operator $\pi_{\lambda}(z,t)$
is surjective, and for all $\varphi,\psi\in L^{2}(\mathbb{R}^{n})$,
we have
\[
\langle\pi_{\lambda}(z,t)\varphi,\pi_{\lambda}(z,t)\psi\rangle=\langle\varphi,\psi\rangle.
\]
In particular, $\pi_{\lambda}(z,t)$ is unitary.
\item The map $\pi_{\lambda}$ is a group homomorphism, i.e.,
\[
\pi_{\lambda}\left(\left(z,t\right)\circ\left(w,s\right)\right)=\pi_{\lambda}\left(z,t\right)\pi_{\lambda}\left(w,s\right).
\]
\item The representation $\pi_{\lambda}$ is strongly continuous with respect
to $(z,t)$, i.e., for every $\varphi\in L^2(\mathbb R^n)$, the map $(z,t)\mapsto \pi_\lambda(z,t)\varphi$ is continuous as a map from $\mathbb H^n$ to $\mathbb L^2(\mathbb R^n)$.
\end{enumerate}
\begin{defn}
Given $f\in L^{1}\left(\mathbb{H}^{n}\right)$, for every $\lambda\neq0$,
the group Fourier transform $\widehat{f}\left(\lambda\right)$  of $f$ is
the operator acting on $L^{2}\left(\mathbb{R}^{n}\right)$ defined
by
\[
\widehat{f}\left(\lambda\right)=\int_{\mathbb{H}^{n}}f\left(z,t\right)\pi_{\lambda}\left(z,t\right)dzdt
\]
where the integration is in the Bochner sense.
\end{defn}

\begin{rem}
The group Fourier transform can be naturally extended to a finite
measure $\sigma$ by 
\[
\widehat{\sigma}\left(\lambda\right)=\int_{\mathbb{H}^{n}}\pi_{\lambda}\left(z,t\right)d\sigma\left(z,t\right).
\]
We can also write the group Fourier transform in a slightly different
way. For $z\in\mathbb{C}^{n}$, define
\[
f^{\lambda}\left(z\right)=\int_{-\infty}^{+\infty}f\left(z,t\right)e^{i\lambda t}dt.
\]
Then, for every $\varphi\in L^{2}\left(\mathbb{R}^{n}\right)$, we
have
\begin{align*}
\widehat{f}\left(\lambda\right)\varphi\left(\xi\right)= & \int_{\mathbb{C}^{n}}\int_{-\infty}^{+\infty}f\left(z,t\right)e^{i\lambda t}\pi_{\lambda}\left(z\right)\varphi\left(\xi\right)dzdt\\
= & \int_{\mathbb{C}^{n}}f^{\lambda}\left(z\right)\pi_{\lambda}\left(z\right)\varphi\left(\xi\right)dz.
\end{align*}
\end{rem}

Let $f,g\in L^{1}\left(\mathbb{H}^{n}\right)$ their convolution is
defined by
\begin{align*}
f*g\left(z,t\right) & =\int_{\mathbb{H}^{n}}f\left(\rho,\tau\right)g\left(\left(\rho,\tau\right)^{-1}\circ\left(z,t\right)\right)d\rho d\tau.
\end{align*}

\noindent We have the following property.
\begin{thm}
\label{thm:trasf conv}Let $f,g\in L^{1}\left(\mathbb{H}^{n}\right)$. Then,
\[
\widehat{f*g}\left(\lambda\right)=\widehat{f}\left(\lambda\right)\widehat{g}\left(\lambda\right).
\]
\end{thm}

\subsection{Plancherel's theorem}
The group Fourier transform just introduced satisfies a Plancherel-type identity. In order to state it, we
need a few definitions. Let $\mathrm{HS}$ denote the Hilbert
space of Hilbert-Schmidt operators on $L^{2}\left(\mathbb{R}^{n}\right)$. Its inner product is defined by
\[
\left\langle S,T\right\rangle =\text{tr}\left(T^{*}S\right)=\sum_{j\in J}\left\langle T^{*}Se_{j},e_{j}\right\rangle 
\]
where $\left\{ e_{j}\right\} _{j\in J}$ is any orthonormal basis of $L^2(\mathbb R^n)$.
In particular, observe that 
\begin{align*}
\left\Vert T\right\Vert _{\mathrm{HS}}^{2}=\left\langle T,T\right\rangle  & =\sum_{j\in J}\left\langle T^{*}Te_{j},e_{j}\right\rangle =\sum_{j\in J}\left\langle Te_{j},Te_{j}\right\rangle =\sum_{j\in J}\left\Vert Te_{j}\right\Vert ^{2}.
\end{align*}
Let $\mathbb R^*=\mathbb R\backslash\{0\}$ and let $L^{2}\left(\mathbb{R}^{*},\mathrm{HS}\right)$ be the space of functions
$F:\mathbb{R}^{*}\to \mathrm{HS}$ such that
\[
\left\Vert F\right\Vert ^{2}=\frac{1}{\left(2\pi\right)^{n+1}}\int_{\mathbb{R}^*}\left\Vert F\left(\lambda\right)\right\Vert _{\mathrm{HS}}^{2}\left|\lambda\right|^{n}d\lambda<+\infty.
\]
We have the following result.
\begin{thm}
\label{thm:Plancherel}The group Fourier transform extends to an isometry
between $L^{2}\left(\mathbb{H}^{n}\right)$ and $L^{2}\left(\mathbb{R}^{*},\mathrm{HS}\right)$.
More precisely, we have
\[
\frac{1}{\left(2\pi\right)^{n+1}}\int_{\mathbb{R}^*}\big\Vert \widehat{f}\left(\lambda\right)\big\Vert _{\mathrm{HS}}^{2}\left|\lambda\right|^{n}d\lambda=\int_{\mathbb{H}^{n}}\left|f\left(x,y,t\right)\right|^{2}dxdydt.
\]
\end{thm}

\section{The Fourier transform of cylindrically radial functions}\label{s:transform_radial}
In this section we recall the results concerning the Fourier transform of (cylindrically) radial functions. To this end, we first introduce Laguerre and Hermite polynomials and functions. We refer the reader to \cite{T_book} for details.

\subsection{Generalized Laguerre polynomials and functions}
For every $\delta\geqslant0$ and $k\in\mathbb{N}$ the generalized
Laguerre polynomials of type $\delta$, are defined by
\[
L_{k}^{\delta}\left(t\right)=t^{-\delta}e^{t}\frac{1}{k!}\frac{d^{k}}{dt^{k}}\left(e^{-t}t^{k+\delta}\right).
\]
When $\delta=0$ we will also write $L_{k}\left(t\right)$ in place
of $L_{k}^{0}\left(t\right)$. We also define the Laguerre functions
by 
\[
\Lambda_{k}^{\delta}\left(t\right)=r^\delta_k L_{k}^{\delta}\left(t\right)e^{-\frac{1}{2}t}t^{\delta/2}.
\]
where
\begin{equation}\label{eq:r_k}
r_{k}^{\delta}=\left(\frac{k!}{\Gamma\left(k+\delta+1\right)}\right)^{1/2}.
\end{equation}
When $\delta=n-1$, we will also use the following rescaled Laguerre
functions defined in $\mathbb{C}^{n}$ by
\begin{equation}
\varphi_{k}^{\lambda}\left(z\right)=\left(\frac{\left|\lambda\right|}{2\pi}\right)^{n}L_{k}^{n-1}\left(\frac{1}{2}\left|\lambda\right|\left|z\right|^{2}\right)e^{-\frac{1}{4}\left|\lambda\right|\left|z\right|^{2}}.\label{eq:def phi lambda k}
\end{equation}

\subsection{Hermite polynomials and functions}

For every $k\in\mathbb{N}$, Hermite polynomials $H_{k}\left(t\right)$
are defined by
\[
H_{k}\left(t\right)=\left(-1\right)^{k}e^{t^{2}}\frac{d^{k}}{dt^{k}}\left\{ e^{-t^{2}}\right\} .
\]
It is immediate to show that $H_{k}\left(t\right)$ is a polynomial
of degree $k$. The normalized Hermite functions are defined by
\[
h_{k}\left(t\right)=\frac{1}{\sqrt{2^{k}k!\sqrt{\pi}}}H_{k}\left(t\right)e^{-t^{2}/2}.
\]
It is well known that the Hermite functions $h_{k}\left(t\right)$
are a complete orthonormal system in $L^{2}\left(\mathbb{R}\right)$.
We now define Hermite functions in $\mathbb{R}^{n}$ as a tensor product
of one-variable Hermite functions. Let $\alpha\in\mathbb{N}^{n}$,
$x\in\mathbb{R}^{n}$, and let
\[
\Phi_{\alpha}\left(x\right)=\prod_{j=1}^{n}h_{\alpha_{j}}\left(x_{j}\right).
\]
The family $\left\{ \Phi_{\alpha}\right\} _{\alpha\in\mathbb{N}^{n}}$
is a complete orthonormal system in $L^{2}\left(\mathbb{R}^{n}\right)$.

For every $\lambda\neq0$, we also introduce the rescaled Hermite
functions by
\[
\Phi_{\alpha}^{\lambda}\left(x\right)=\left|\lambda\right|^{n/4}\Phi_{\alpha}\left(\left|\lambda\right|^{1/2}x\right).
\]
Again these functions are an orthonormal basis of $L^{2}\left(\mathbb{R}^{n}\right)$.
\begin{defn}
\label{def:Phi_alpha_beta}For every $\alpha,\beta\in\mathbb{N}^{n}$
and $\lambda\neq0$, the special Hermite functions are defined by
\[
\Phi_{\alpha,\beta}^{\lambda}\left(z\right)=\left(\frac{\left|\lambda\right|}{2\pi}\right)^{n/2}\left\langle \pi_{\lambda}\left(z\right)\Phi_{\alpha}^{\lambda},\Phi_{\beta}^{\lambda}\right\rangle .
\]
\end{defn}

We also have the following identity which relates the rescaled Hermite functions and the rescaled Laguerre functions.
\begin{prop}
\label{prop:phi_k}Let $k\in\mathbb{N}$. Then
\[
\sum_{\left|\alpha\right|=k}\left\langle \pi_{\lambda}\left(z\right)\Phi_{\alpha}^{\lambda},\Phi_{\alpha}^{\lambda}\right\rangle =\left(\frac{2\pi}{\left|\lambda\right|}\right)^{n}\varphi_{k}^{\lambda}\left(z\right).
\]
\end{prop}

\subsection{Group Fourier transform of radial functions}
Given $f\in L^{1}\left(\mathbb{H}^{n}\right)$ we say that $f$
is (cylindrically) radial if there exists a function $f_{\#}$ such
that for every $\left(z,t\right)\in\mathbb{H}^{n}$ we have
\[
f\left(z,t\right)=f_{\#}\left(\left|z\right|,t\right).
\]

The group Fourier transform of a cylindrically radial function
diagonalizes on rescaled Hermite functions, see  \cite{Geller}. With a small abuse of
notation we will denote with $\widehat{f}$ also the eigenvalues of
the operator. More precisely we have the following:
\begin{thm}
\label{thm:Fourier transf radial}Let $f\in L^{1}\left(\mathbb{H}^{n}\right)$
and assume that $f\left(z,t\right)=f_{\#}\left(\left|z\right|,t\right)$.
Then
\[
\widehat{f}\left(\lambda\right)\Phi_{\alpha}^{\lambda}=c_{n}\widehat{f}\left(\lambda,\left|\alpha\right|\right)\cdot\Phi_{\alpha}^{\lambda},
\]
where $c_{n}=\text{\ensuremath{\left(2\pi\right)^{n}2^{1-n}}}$,
\begin{align*}
\widehat{f}\left(\lambda,k\right)= & r_{k}\int_{0}^{+\infty}f_{\#}^{\lambda}\left(r\right)\left(\left|\lambda\right|r^{2}\right)^{\frac{1-n}2}\Lambda_{k}^{n-1}\left(\frac{1}{2}\left|\lambda\right|r^{2}\right)r^{2n-1}dr,
\end{align*}
and $r_k=r_k^{n-1}$ as defined in \eqref{eq:r_k}. Moreover, if $f$ is in the Schwartz class, we have the reconstruction formula 
\begin{align}\label{thm: ricostruzione}
\begin{split}
f\left(z,t\right)= & \frac{2^{\frac{1-n}2}}{2\pi}\int_{-\infty}^{+\infty}e^{-it\lambda}\sum_{k=0}^{+\infty}\widehat{f}\left(\lambda,k\right)L_{k}^{n-1}\left(\frac{1}{2}\left|\lambda\right|\left|z\right|^{2}\right)e^{-\frac{1}{4}\left|\lambda\right|\left|z\right|^{2}}\left|\lambda\right|^{n}d\lambda\\
= & 2^{\frac{1-n}2}\left(2\pi\right)^{n-1}\int_{-\infty}^{+\infty}e^{-it\lambda}\sum_{k=0}^{+\infty}\widehat{f}\left(\lambda,k\right)\varphi_{k}^{\lambda}\left(z\right)d\lambda.
\end{split}
\end{align}

\end{thm}

We conclude the section recalling how the Fourier transform interacts with the dilations $D_\rho$. Let $f\in L^{1}\left(\mathbb{H}^{n}\right)$ be a radial function
and for every $\rho>0$, let 
\[
f_{\rho}\left(z,t\right)=f\left(D_{\rho^{-1}}\left(z,t\right)\right)=f_{\#}\left(\rho^{-1}\left|z\right|,\rho^{-2}t\right).
\]
We claim that
\begin{align}
\widehat{f_{\rho}}\left(\lambda,k\right)=\rho^{2n+2}\widehat{f}\left(\rho^{2}\lambda,k\right).
\label{eq:Dilatazioni fourier}
\end{align}
Indeed, since
\begin{align*}
\left(f_{\rho}\right)_{\#}^{\lambda}\left(r\right)= & \int_{-\infty}^{+\infty}\left(f_{\rho}\right)_{\#}\left(r,t\right)e^{i\lambda t}dt=\int_{-\infty}^{+\infty}f_{\#}\left(\rho^{-1}r,\rho^{-2}t\right)e^{i\lambda t}dt\\
= & \rho^{2}\int_{-\infty}^{+\infty}f_{\#}\left(\rho^{-1}r,u\right)e^{i\lambda\rho^{2}u}du=\rho^{2}f_{\#}^{\rho^{2}\lambda}\left(\rho^{-1}r\right),
\end{align*}
we have
\begin{align*}
\widehat{f_{\rho}}\left(\lambda,k\right)= & r_{k}\int_{0}^{+\infty}\rho^{2}f_{\#}^{\rho^{2}\lambda}\left(\rho^{-1}r\right)\left(\left|\lambda\right|r^{2}\right)^{\left(1-n\right)/2}\Lambda_{k}^{n-1}\left(\frac{1}{2}\left|\lambda\right|r^{2}\right)r^{2n-1}dr\nonumber \\
= & r_{k}\int_{0}^{+\infty}\rho^{2n+2}f_{\#}^{\rho^{2}\lambda}\left(s\right)\left(\left|\lambda\right|\rho^{2}s^{2}\right)^{\left(1-n\right)/2}\Lambda_{k}^{n-1}\left(\frac{1}{2}\left|\lambda\right|\rho^{2}s^{2}\right)s^{2n-1}ds\\
= & \rho^{2n+2}\widehat{f}\left(\rho^{2}\lambda,k\right).\nonumber 
\end{align*}

\section{Proof of the Main Result}\label{s:main_result}
Adopting a strategy inspired
by Roth and Montgomery, we reduce the argument of the proof of Theorem \ref{MainResult} to the derivation of lower bounds
for two different contributions, one of geometric nature and one related
to the point distribution. The corresponding estimates, stated in
(\ref{eq: stima I}) and (\ref{eq:stima J}), both require nontrivial
additional work, which constitutes the main content of the rest of the paper.



\begin{proof}[Proof of Theorem \ref{MainResult}]
Let us rewrite the discrepancy for the points $\left\{ \left(\zeta_{j},\tau_{j}\right)\right\} _{j=1}^{N}\subseteq\mathbb{H}^{n}$
in a more convenient form. Since $\chi_{B_{\rho}}\left(w,s\right)=\chi_{B_{\rho}}\left(\left(w,s\right)^{-1}\right)$,
we have
\begin{align*}
D_{N}\left(z,t;\rho\right)= & \sum_{j=1}^{N}\chi_{B_{\rho}}\left(\left(\zeta_{j},\tau_{j}\right)^{-1}\circ\left(z,t\right)\right)-N\int_{\mathbb{H}^{n}}\chi_{B_{\rho}}\left(\left(\zeta,\tau\right)^{-1}\circ\left(z,t\right)\right)d\mu\left(\zeta,\tau\right)\\
= & \left(\sum_{j=1}^{N}\delta_{\left(\zeta_{j},\tau_{j}\right)}\right)*\chi_{B_{\rho}}\left(z,t\right)-N\mu*\chi_{B_{\rho}}\left(z,t\right)=\sigma*\chi_{B_{\rho}}\left(z,t\right)
\end{align*}
where we set 
\[
\sigma=\sum_{j=1}^{N}\delta_{\left(\zeta_{j},\tau_{j}\right)}-N\mu
\]
and $\delta$ denotes the Dirac measure. Using Theorem \ref{thm:Plancherel},
Theorem \ref{thm:trasf conv}, the fact that $\chi_{B\rho}$ is radial,
and Theorem \ref{thm:Fourier transf radial}, we get
\begin{align*}
\int_{\mathbb{H}^{n}}\left|D_{N}\left(z,t;\rho\right)\right|^{2}dzdt\simeq & \int_{-\infty}^{+\infty}\left\Vert \widehat{\sigma}\left(\lambda\right)\widehat{\chi_{B_{\rho}}}\left(\lambda\right)\right\Vert _{\mathrm{HS}}^{2}\left|\lambda\right|^{n}d\lambda\\
= & \int_{-\infty}^{+\infty}\sum_{\alpha\in\mathbb{N}^{n}}\left\Vert \widehat{\sigma}\left(\lambda\right)\widehat{\chi_{B_{\rho}}}\left(\lambda\right)\Phi_{\alpha}^{\lambda}\right\Vert _{L^{2}\left(\mathbb{R}^{n}\right)}^{2}\left|\lambda\right|^{n}d\lambda\\
\simeq & \int_{-\infty}^{+\infty}\sum_{\alpha\in\mathbb{N}^{n}}\left|\widehat{\chi_{B_{\rho}}}\left(\lambda,\left|\alpha\right|\right)\right|^{2}\left\Vert \widehat{\sigma}\left(\lambda\right)\Phi_{\alpha}^{\lambda}\right\Vert _{L^{2}\left(\mathbb{R}^{n}\right)}^{2}\left|\lambda\right|^{n}d\lambda\\
= & \int_{-\infty}^{+\infty}\sum_{k=0}^{+\infty}\left|\widehat{\chi_{B_{\rho}}}\left(\lambda,k\right)\right|^{2}\sum_{\left|\alpha\right|=k}\left\Vert \widehat{\sigma}\left(\lambda\right)\Phi_{\alpha}^{\lambda}\right\Vert _{L^{2}\left(\mathbb{R}^{n}\right)}^{2}\left|\lambda\right|^{n}d\lambda.
\end{align*}
From now on, for every $k\in\mathbb{N}$, we set \[\nu=\nu_k=4k+2n.\] 
Let $\Lambda>0$ and let
\[
F_{\Lambda}=\left\{ \left(k,\lambda\right)\in\mathbb{N}\times\mathbb{R}^{*}:\left\langle \lambda\nu\right\rangle \leqslant\Lambda,\left|\lambda\right|\leqslant1\right\} .
\]
For every $0<s<1$ such that $s\Lambda>Q-1$, we
consider
\begin{align}
 & \int_{0}^{1}\int_{\mathbb{H}^{n}}\left|D_{N}\left(z,t;\rho\right)\right|^{2}dzdtd\rho\nonumber \\
\gtrsim & \int_{-\infty}^{+\infty}\sum_{k=0}^{+\infty}\chi_{F_{\Lambda}}\left(k,\lambda\right)\int_{0}^{1}\left|\widehat{\chi_{B_{\rho}}}\left(\lambda,k\right)\right|^{2}d\rho\sum_{\left|\alpha\right|=k}\left\Vert \widehat{\sigma}\left(\lambda\right)\Phi_{\alpha}^{\lambda}\right\Vert _{L^{2}\left(\mathbb{R}^{n}\right)}^{2}\left|\lambda\right|^{n}d\lambda\nonumber \\
\geqslant & \left(\inf_{\left(k,\lambda\right)\in F_{\Lambda}}e^{\frac{1}{2}\nu\left|\lambda\right|s}\int_{0}^{1}\left|\widehat{\chi_{B_{\rho}}}\left(\lambda,k\right)\right|^{2}d\rho\right)\nonumber \\
 & \times\sum_{k=0}^{+\infty}\int_{-\infty}^{+\infty}\chi_{F_{\Lambda}}\left(k,\lambda\right)e^{-\frac{1}{2}\nu\left|\lambda\right|s}\sum_{\left|\alpha\right|=k}\left\Vert \widehat{\sigma}\left(\lambda\right)\Phi_{\alpha}^{\lambda}\right\Vert _{L^{2}\left(\mathbb{R}^{n}\right)}^{2}\left|\lambda\right|^{n}d\lambda\nonumber \\
= & \mathcal{I}\left(\Lambda,s\right)\cdot\mathcal{J}\left(\Lambda,s\right).\label{eq:I x J}
\end{align}
Observe that, while the term $\mathcal{I}$ only depends on $\Lambda$
and $s$, the term $\mathcal{J}$ also depends on $\sigma$, that
is on the sequence of points and on the measure $\mu$. In Section
\ref{sec:Estimate I} and \ref{sec:Estimate J} we will prove that
\begin{equation}
\mathcal{I}\left(\Lambda,s\right)\gtrsim s^{\frac{Q-1}{2}}\label{eq: stima I}
\end{equation}
and 
\begin{equation}
\mathcal{J}\left(\Lambda,s\right)\gtrsim Ns^{1-\frac{Q}{2}}-C_{1}N^{2}-C_{2}N^{2}s^{-1}\Lambda^{\frac{Q-4}2}e^{-\frac{1}{2}s\Lambda}.\label{eq:stima J}
\end{equation}
Using these estimates we obtain
\begin{align*}
\int_{0}^{1}\int_{\mathbb{H}^{n}}\left|D_{N}\left(z,t;\rho\right)\right|^{2}dzdtd\rho\gtrsim & Ns^{1/2}\left(1-C_{1}Ns^{\frac{Q}{2}-1}-C_{2}N\left(s\Lambda\right)^{\frac{Q-4}{2}}e^{-\frac{1}{2}s\Lambda}\right).
\end{align*}
Setting $s=\left(2C_{1}N\right)^{-1/n}$ and $s\Lambda=3\log N$
yields
\begin{align*}
\int_{0}^{1}\int_{\mathbb{H}^{n}}\left|D_{N}\left(z,t;\rho\right)\right|^{2}dzdtd\rho\gtrsim & N^{1-\frac{1}{Q-2}}\left(\frac{1}{2}-C_{2}N^{-1/2}\left(3\log N\right)^{\frac{Q-4}{2}}\right)\gtrsim N^{1-\frac{1}{Q-2}},
\end{align*}
for $N$ large enough.
\end{proof}

\section{An upper bound for the minimal quadratic discrepancy}\label{sec:upper bound}
In this section we apply a general result for discrepancy contained in \cite[§8]{BCCGT2019} to our setting. Let $F\subseteq\mathbb{H}^n$ have Lebesgue measure $1$, and let $\mu$ be the probability measure obtained by restricting the Lebesgue measure to $F$.
Assume that $F$ is chosen so that $\mu$  is lower and upper Ahlfors regular, namely, that there exist positive constants $c_1,c_2$ such that for every $0<r<\textrm{diam}(F)$ and every ball $\mathfrak{B}_{r}$  centered in $F$ we have
\begin{align*}
c_1r^{Q}\leqslant\mu\left(\mathfrak{B}_{r}\right)\leqslant c_2r^{Q}.
\end{align*}
By a general result of Gigante and Leopardi \cite[Theorem 2]{GL2017}, if $N$ is a sufficiently large integer, then there exists a partition of $F=\bigcup\limits_{j=1,\ldots,N}F_j$ such that $\mu(F_j)=1/N$ and $\textrm{diam}(F_j)\approx N^{-1/Q}$.
We also need the following technical lemma.
\begin{lem}
Let $0<s<\rho$. Then
\[
\left\{ \left(z,t\right)\in B_{\rho}:\mathrm{dist}\left(\left(z,t\right),F\setminus B_{\rho}\right)\leqslant s\right\} \cup\left\{ \left(z,t\right)\in F\setminus B_{\rho}:\mathrm{dist}\left(\left(z,t\right),B_{\rho}\right)\leqslant s\right\} \subseteq B_{\rho+s}\setminus B_{\rho-s}.
\]
\end{lem}

\begin{proof}
We start by noting that if $a>0$ and $\left(w,\tau\right)\in B_{a}$
then $\overline{\mathfrak{B}_{s}\left(w,\tau\right)}\subseteq B_{a+s}$. Indeed,
let $\left(z,t\right)\in\overline{\mathfrak{B}_{s}\left(w,\tau\right)}$. Hence
$\left(z,t\right)=\left(w,\tau\right)\circ\left(y,\sigma\right)$
for some $\left(y,\sigma\right)$ such that $\left\Vert \left(y,\sigma\right)\right\Vert\leqslant s$.
In particular $\left|y\right| \leqslant s$ and $\left|\sigma\right| \leqslant s^{2}$.
Therefore
\[
\left|z\right|\leqslant\left|w\right|+\left|y\right|\leqslant a+s
\]
and
\begin{align*}
\left|t\right|\leqslant & \left|\tau\right|+\left|\sigma\right|+\frac{1}{2}\left|w\right|\left|y\right|\leqslant a^{2}+s^{2}+\frac{1}{2}as\leqslant\left(a+s\right)^{2}.
\end{align*}
It follows that $\left(z,t\right)\in B_{a+s}$. 
We now show that
\begin{align}
\left\{ \left(z,t\right)\in B_{\rho}:\mathrm{dist}\left(\left(z,t\right),F\setminus B_{\rho}\right)\leqslant s\right\}  
\subseteq B_{\rho}\setminus B_{\rho-s}.
\label{eq:prima_inclusione}
\end{align}
Assume, by contradiction, that there exists $\left(w,\tau\right)\in B_{\rho-s}$
such that $\text{dist}\left(\left(w,\tau\right),F\setminus B_{\rho}\right)\leqslant s$.
Hence $\left(F\setminus B_{\rho}\right)\cap\overline{\mathfrak{B}_{s}\left(w,\tau\right)}\neq\emptyset$.
This is not possible since $\overline{\mathfrak{B}_{s}\left(w,\tau\right)}\subseteq B_{\rho}$ and \eqref{eq:prima_inclusione} follows.
 On the other hand, if $\left(w,\tau\right)\in B_{\rho}$ we have $\overline{\mathfrak{B}_{s}\left(w,\tau\right)}\subseteq B_{\rho+s}$.
 Therefore,
\begin{align*}
\left\{ \left(z,t\right)\in F\setminus B_{\rho}:\text{dist}\left(x,B_{\rho}\right)\leqslant s\right\} \subseteq & \bigcup_{\left(w,\tau\right)\in B_{\rho}}\overline{\mathfrak{B}_{s}\left(w,\tau\right)}\setminus B_{\rho}\subseteq B_{\rho+s}\setminus B_{\rho}.
\end{align*}
\end{proof}

We now estimate the discrepancy with respect to the family of sets $B_\rho(z,t)=(z,t)\circ B_\rho$ with  $\rho<\textrm{diam}(F)$ and $(z,t)\in F$. 
Using the previous lemma, the translation invariance of the Lebesgue measure and the translation invariance of the distance we obtain
\begin{align*}
 \left| \left\{ x\in B_\rho(z,t) :\text {dist}\left\{ x,F\setminus B_\rho(z,t)\right\} \leqslant s\right\} \right|
+ \left| \left\{ x\in F \setminus B_\rho(z,t):\text {dist}\left\{ x,B_\rho(z,t)\right\} \leqslant s\right\} \right| 
\leqslant & |B_{\rho+s}|-|B_{\rho-s}| \\
\leqslant & c\,s
\end{align*}
with a constant $c$ independent of $\rho$ as long as $\rho<\textrm{diam}(F)$.
By Corollary 8.2 in \cite{BCCGT2019}, for $N$ sufficiently large there exists a finite sequence of points such that
\begin{align*}
    \int_0^{\textrm{diam}(F)} \int_F \left|D_{N}\left(z,t;\rho\right)\right|^{2}dzdt\,d\rho
    \leqslant CN^{1-\frac{1}{Q}}.
\end{align*}

\section{Estimates for Laguerre functions}\label{s:estimates_laguerre}

This section collects some technical ingredients needed for the analysis
of the asymptotic behavior of $\widehat{\chi_{B}}$. In particular,
we rely on classical estimates for Laguerre polynomials, and consequently
for Laguerre functions (see \cite{AskeyWainger1965}, \cite{Erdelyi1960}
and \cite{FrenzenWong1988}).

\subsection{The Bessel and the Airy functions}

In the following, $J_{m}$ and $Y_{m}$ denote the usual Bessel functions
of the first and the second kind. We shall use the standard asymptotics
\begin{align}
J_{m}\left(z\right) & \sim\frac{z^{m}}{2^{m}m!}, & \text{for } & z\to0^{+},\label{eq: bessel picc}\\
J_{m}\left(z\right) & =\sqrt{\frac{2}{\pi z}}\cos\left(z-\frac{m\pi}{2}-\frac{\pi}{4}\right)+O\left(\frac{1}{z^{3/2}}\right), & \text{for } & z\to+\infty.\label{eq:bessel gra}
\end{align}
In particular (\ref{eq: bessel picc}) follows from the Taylor expansion
of $J_{m}\left(z\right)$ and the location of its first zero. 

Furthermore, we define
\begin{equation}
\widetilde{J}_{m}\left(z\right)=\begin{cases}
J_{m}\left(z\right) & \text{for }z\text{ small,}\\
\left(\left|J_{m}\left(z\right)\right|^{2}+\left|Y_{m}\left(z\right)\right|^{2}\right)^{1/2} & \text{for }z\text{ large.}
\end{cases}\label{eq: def J tilde}
\end{equation}
Observe that 
\begin{equation}
\widetilde{J}_{m}\left(z\right)\approx\frac{z^{m}}{\left\langle z\right\rangle ^{m+1/2}}.\label{eq:J_tilde asympt}
\end{equation}
Indeed, the case $z$ small is just (\ref{eq: bessel picc}) while
the case $z$ large follows from the fact that $\widetilde{J}_{m}\left(z\right)$
equals the modulus of the Hankel function $H_{m}^{\left(1\right)}$.

We denote by $\text{Ai}$ and ${\rm Bi}$ the Airy function of the
first and the second kind. The estimate of $\text{Ai}$ that we shall
use are
\begin{align}
\mathrm{Ai}\left(u\right) & =\frac{1}{\pi^{1/2}\left(-u\right)^{1/4}}\cos\left(\frac{2}{3}\left(-u\right)^{3/2}-\frac{\pi}{4}\right)+O\left(\left(-u\right)^{-7/4}\right), & \text{for } & u\to-\infty,\nonumber \\
\mathrm{Ai}\left(u\right) & =O\left(1\right), & \text{for } & |u|\to0,\label{eq:Stime Airy}\\
\mathrm{Ai}\left(u\right) & =\frac{1}{2\sqrt{\pi}u^{1/4}}e^{-\frac{1}{2}u^{3/2}}\left(1+O\left(u^{-3/2}\right)\right) & \text{for } & u\to+\infty\nonumber \\
\mathrm{Ai}'\left(u\right) & =-\frac{u^{1/4}}{2\sqrt{\pi}}e^{-\frac{2}{3}u^{3/2}}\left(1+O\left(u^{-3/2}\right)\right) & \text{for } & u\to+\infty.\nonumber 
\end{align}
We also introduce the auxiliary function
\[
\widetilde{\mathrm{Ai}}\left(z\right)=\begin{cases}
\left(\left|\textrm{Ai}\left(z\right)\right|^{2}+\left|\textrm{Bi}\left(z\right)\right|^{2}\right)^{1/2} & \text{for }z<0,\\
\mathrm{Ai}\left(z\right) & \text{for }z\geqslant0.
\end{cases}
\]
From (\ref{eq:Stime Airy}) and similar asymptotics for ${\rm Bi}$,
we have
\[
\widetilde{{\rm Ai}}\left(u\right)=\begin{cases}
O\left(\left(-u\right)^{-1/4}\right), & \text{for }u\to-\infty,\\
O\left(1\right), & \text{for }|u|\to0.
\end{cases}
\]
For later use, we introduce the functions
\[
\mathrm{IAi}\left(u\right)=\int_{-\infty}^{u}\mathrm{Ai}\left(\tau\right)d\tau,\qquad\mathrm{IIAi}\left(u\right)=\int_{-\infty}^{u}\mathrm{IAi}\left(\tau\right)d\tau,\qquad u\in\mathbb{R},
\]
and we collect some of their estimates of the next Lemma.
\begin{lem}
\label{lem: stime IA}The following estimates hold
\begin{align*}
\mathrm{IAi}\left(u\right) & =\begin{cases}
\frac{1}{\pi^{1/2}\left(-u\right)^{3/4}}\cos\left(\frac{2}{3}\left(-u\right)^{3/2}+\frac{\pi}{4}\right)+O\left(\left(-u\right)^{-9/4}\right), & \text{for }u\to-\infty,\\
\frac{2}{3}\left(1+O\left(u\right)\right), & \text{for }|u|\to0.
\end{cases}\\
\mathrm{IIAi}\left(u\right) & =\begin{cases}
O\left(\left(-u\right)^{-5/4}\right), & \text{for }u\to-\infty,\\
O\left(1\right), & \text{for }|u|\to0.
\end{cases}
\end{align*}
\end{lem}

\begin{proof}
The estimates for $\mathrm{IAi}$ are well known, see e.g. \cite[9.7.9, 9.10.6 and 9.10.11]{NIST}.
The estimates for $\mathrm{IIAi}$ follows integrating by parts the
estimates for ${\rm IAi}$. Note that $\mathrm{IIAi}\left(0\right)$
can be computed using \cite[9.10.22]{NIST}.
\end{proof}
The following lemma will be useful in the next section.
\begin{lem}
\label{lem:Bessel^2}The estimates from below on the following integrals
hold.
\begin{enumerate}
\item There exists $c>0$ such that if $1<a<b-1$, then
\[
\int_{a}^{b}\left|x^{1/2}J_{m}\left(x\right)\right|^{2}dx>c\left(b-a\right).
\]
\item There exist $\tilde{c},c>0$ such that whenever $a+1<b<-\tilde{c}$,
we have
\[
\int_{a}^{b}\left|\mathrm{IAi}\left(x\right)\right|^{2}dx>\frac{c}{\sqrt{-b}}.
\]
\end{enumerate}
\end{lem}

\begin{proof}
If $a$ is sufficiently large, say $a\geqslant C$, by (\ref{eq:bessel gra}),
it is easy to check that
\[
\int_{a}^{b}\left|x^{1/2}J_{m}\left(x\right)\right|^{2}dx\gtrsim\int_{a}^{b}\left|\cos\left(x-\frac{\pi m}{2}-\frac{\pi}{4}\right)\right|^{2}dx+O\left(\frac{1}{a}\right)\gtrsim b-a.
\]
Let now $b<C$ and let $\mathcal{Z}=\left\{ z\in\left[1,C\right]\colon J_{m}\left(z\right)=0\right\} $
and let $W_{\varepsilon}=\left\{ x\in\left[1,C\right]\colon d\left(x,\mathcal{Z}\right)\geqslant\varepsilon\right\} .$
Since $\mathcal{Z}$ is finite one can easily find a positive $\varepsilon$,
independent of $a$ and $b$ such that 
\[
\left|W_{\varepsilon}\cap\left[a,b\right]\right|\geqslant\frac{b-a}{2}.
\]
Let $c={\displaystyle \min_{x\in W_{\varepsilon}}}\left|x^{1/2}J_{m}\left(x\right)\right|$
then 
\[
\int_{a}^{b}\left|x^{1/2}J_{m}\left(x\right)\right|^{2}dx=\int_{\left[a,b\right]\cap W_{\varepsilon}}\left|x^{1/2}J_{m}\left(x\right)\right|^{2}dx\gtrsim c^{2}\left|\left[a,b\right]\cap W_{\varepsilon}\right|\geqslant\frac{c^{2}}{2}\left(b-a\right).
\]
The case $a<C<b$ follows combining the previous two cases.

For $\tilde{c}>0$ large enough, by Lemma \ref{lem: stime IA}, we
have 
\begin{align*}
\int_{a}^{b}\left|\mathrm{IAi}\left(x\right)\right|^{2}dx\gtrsim & \int_{\left[a,b\right]}\left(-x\right)^{-3/2}\left|\cos\left(\frac{2}{3}\left(-x\right)^{3/2}+\frac{\pi}{4}\right)\right|^{2}dx+O\left(\int_{a}^{b}\left(-x\right)^{-9/2}dx\right)\\
\gtrsim & \int_{\frac{2}{3}\left(-b\right)^{3/2}}^{\frac{2}{3}\left(-a\right)^{3/2}}u^{-4/3}\left|\cos\left(u+\frac{\pi}{4}\right)\right|^{2}du+O\left(\left(-b\right)^{-7/2}\right)
\end{align*}
Let $\varepsilon$ be a positive small number. Let
\[
V_{\varepsilon}=\left\{ u\in\mathbb{R}\colon d\left(u,\frac{\pi}{4}+\mathbb{Z}\pi\right)>\varepsilon\right\} ,
\]
and observe that that there exists $c_{\varepsilon}>0$ such that
$\left|\cos\left(u+\frac{\pi}{4}\right)\right|^{2}>c_{\varepsilon}$,
for every $u\in V_{\varepsilon}$. Hence
\begin{align*}
\int_{a}^{b}\left|\mathrm{IAi}\left(x\right)\right|^{2}dx & \gtrsim\int_{\frac{2}{3}\left[\left(-b\right)^{3/2},\left(-a\right)^{3/2}\right]\cap V_{\varepsilon}}u^{-4/3}du+O\left(\left(-b\right)^{-7/2}\right)\\
 & \gtrsim\int_{\frac{2}{3}\left(-b\right)^{3/2}}^{\frac{2}{3}\left(-a\right)^{3/2}}u^{-4/3}du+O\left(\left(-b\right)^{-7/2}\right)\\
 & \gtrsim\frac{1}{\sqrt{-b}}+O\left(\frac{1}{\left(-b\right)^{7/2}}\right)\gtrsim\frac{1}{\sqrt{-b}}.
\end{align*}
\end{proof}

\subsection{The functions $A$ and $\Theta$}

Following \cite{Erdelyi1960,FrenzenWong1988}, we introduce two auxiliary
functions that will play a key role in the estimates for Laguerre
functions, and we recall their main properties. For $0\leqslant t<1$,
we set
\begin{align*}
A\left(t\right) & =\frac{1}{2}\left(\arcsin\sqrt{t}+\sqrt{t-t^{2}}\right)
\end{align*}
and, for $t\geqslant0$,
\[
\Theta\left(t\right)=\begin{cases}
-{\displaystyle \left(\frac{3}{4}\left(\arccos\sqrt{t}-\sqrt{t-t^{2}}\right)\right)^{2/3},} & \text{if }0\leqslant t<1,\\
{\displaystyle \left(\frac{3}{4}\left(\sqrt{t^{2}-t}-\mathrm{arccosh}\sqrt{t}\right)\right)^{2/3},} & \text{if }t\geqslant1.
\end{cases}
\]
These functions play a central role in the asymptotic expansion of
Laguerre functions. We now gather their key properties in the next
lemma.
\begin{lem}
\label{lem: A beta alpha}The function $A$ is smooth in $\left(0,1\right)$
and has the following properties:
\begin{align}
A'\left(t\right) & =\frac{1}{2}\left(\frac{1-t}{t}\right)^{1/2}\label{eq: A'}\\
\frac{\pi}{4}\sqrt{t} & \leqslant A\left(t\right)\leqslant\sqrt{t}\label{eq:Bound A(t)}\\
A\left(t\right) & =t^{1/2}\left(1+O\left(t\right)\right),\qquad t\to0^{+}.\label{eq:asympt A}
\end{align}
The function $\Theta$ is smooth in $\left(0,\infty\right)$, $\Theta'\left(t\right)>0$,
$\Theta\left(1\right)=0$, and
\begin{align}
\Theta\left(t\right) & =2^{-2/3}\left(t-1\right)\left(1+O\left(t-1\right)\right), & t\to1.\label{eq:That t->1}
\end{align}
For $\varepsilon>0$, there exists $C_{\varepsilon}>0$ such that
for every $t\geqslant1+\varepsilon$, we get
\begin{equation}
C_{\varepsilon}t^{2/3}\leqslant\Theta\left(t\right)\leqslant\left(\frac{3}{4}t\right)^{2/3}.\label{eq: theta grande}
\end{equation}
\end{lem}

\begin{proof}
While (\ref{eq: A'}) is straightforward, (\ref{eq:Bound A(t)}) follows
from the fact that the function $A\left(t\right)t^{-1/2}$ is decreasing.
The asymptotic (\ref{eq:asympt A}) follows from the definition of
$A$. From the definition of $\Theta$ it is easy to check that for
$t\neq1$, $\Theta$ is smooth and $\Theta'\left(t\right)>0$. The
only part that requires some effort is the smoothness of $\Theta$
for $t=1$. We start from the Taylor expansion
\[
\arcsin z=\sum_{n=0}^{+\infty}a_{n}z^{2n+1}
\]
with
\[
a_{0}=1\text{ and }a_{n}=\frac{1\cdot3\cdots\left(2n-1\right)}{2\cdot4\cdots\left(2n\right)}\frac{1}{\left(2n+1\right)}.
\]
Then for $0\leqslant t<1$, setting $b_{n}=\binom{1/2}{n}$, we expand
$\Theta\left(t\right)$ as follows:
\begin{align*}
\Theta\left(t\right)= & -\left(\frac{3}{4}\left(\sum_{n=0}^{+\infty}a_{n}\left(\sqrt{1-t}\right)^{2n+1}-\sqrt{1+\left(t-1\right)}\left(1-t\right)^{1/2}\right)\right)^{2/3}\\
= & -\left(\frac{3}{4}\left(1-t\right)^{1/2}\left[\sum_{n=0}^{+\infty}a_{n}\left(1-t\right)^{n}-\sum_{n=0}^{+\infty}b_{n}\left(t-1\right)^{n}\right]\right)^{2/3}\\
= & -\left(\frac{3}{4}\right)^{2/3}\left(1-t\right)^{1/3}F\left(t\right).
\end{align*}
The Taylor expansion of ${\rm arcsinh}$ is closely related to that
of $\arcsin$. In fact,
\[
{\rm arcsinh}\left(z\right)=\sum_{n=0}^{+\infty}\left(-1\right)^{n}a_{n}z^{2n+1}.
\]
For $t\geqslant1$, since $\text{arccosh}\sqrt{t}=\text{arcsinh}\sqrt{t-1}$,
we obtain
\begin{align*}
\Theta\left(t\right)= & \left(\frac{3}{4}\left(\sqrt{1+\left(t-1\right)}\sqrt{t-1}-\sum_{n=0}^{+\infty}\left(-1\right)^{n}a_{n}\left(\sqrt{t-1}\right)^{2n+1}\right)\right)^{2/3}\\
= & \left(\frac{3}{4}\left(t-1\right)^{1/2}\left(\sum_{n=0}^{+\infty}b_{n}\left(t-1\right)^{n}-\sum_{n=0}^{+\infty}a_{n}\left(1-t\right)^{n}\right)\right)^{2/3}\\
= & \left(\frac{3}{4}\right)^{2/3}\left(t-1\right)^{1/3}F\left(t\right).
\end{align*}
It follows that
\begin{align*}
\Theta\left(t\right)= & \left(\frac{3}{4}\right)^{2/3}\left(t-1\right)^{1/3}F\left(t\right)
\end{align*}
Since $a_{0}=1$, $a_{1}=1/6$, $b_{0}=1$ and, $b_{1}=1/2$ we have
\begin{align*}
\sum_{n=0}^{+\infty}a_{n}\left(1-t\right)^{n}-\sum_{n=0}^{+\infty}b_{n}\left(t-1\right)^{n}=-\left(t-1\right)\left[\frac{2}{3}+\sum_{n=1}^{+\infty}c_{n}\left(t-1\right)^{n}\right]
\end{align*}
and therefore
\[
\Theta\left(t\right)=\left(t-1\right)\left(\frac{1}{2}+\frac{3}{4}\sum_{n=1}^{+\infty}c_{n}\left(t-1\right)^{n}\right)^{2/3}.
\]
In particular $\Theta\left(t\right)$ is smooth in a neighborhood
of $t=1$ and $\Theta'\left(1\right)>0$. Observe that also (\ref{eq:That t->1})
follows. Finally, it is easy to check that for $t\geqslant1+\varepsilon$
we have 
\[
\frac{d}{dt}\left(\frac{\Theta\left(t\right)^{3/2}}{t}\right)=\frac{3}{4}\frac{\mathrm{arccosh}\sqrt{t}}{t^{2}}>0,
\]
and then (\ref{eq: theta grande}) follows.
\end{proof}

\subsection{The estimates for Laguerre polynomials of Frenzen and Wong}

In this paragraph we recall the uniform asymptotic estimates for Laguerre
functions due to Frenzen and Wong \cite{FrenzenWong1988}. These formulas
describe the oscillatory and exponential regimes of these functions
in terms of Bessel and Airy functions, and will serve as a key tool
in our analysis.

We recall that 
\[
\nu=4k+2n.
\]
In the following, we will use $k$ and $\nu$ interchangeably
whenever no confusion can arise. We also notice that 
\begin{equation}
r_{k}^{2}=\frac{k!}{\left(n+k-1\right)!}=\frac{1}{\left(k+1\right)\cdots\left(k+n-1\right)}\approx\nu^{1-n}.\label{eq: rk}
\end{equation}

\begin{thm}
\label{thm: FW}The following estimates hold uniformly over the respective
intervals.
\end{thm}

\begin{enumerate}
\item \label{enu:1 est}(Bessel regime) There exist functions $\alpha_{0},$
$\beta_{1}\in C^{\infty}\left(0,1\right)$ such that, for any $a\in\left(0,1\right)$,
and for $0\leqslant x\leqslant a\nu$, (with the notation $t=x/\nu$),
one has
\[
\Lambda_{k}^{n-1}\left(x\right)=2^{1-n}r_{k}x^{\frac{n-1}{2}}\left[\alpha_{0}\left(t\right)\frac{J_{n-1}\left(\nu A\left(t\right)\right)}{A\left(t\right)^{n-1}}-\beta_{1}\left(t\right)\frac{2}{\nu}\frac{J_{n}\left(\nu A\left(t\right)\right)}{A\left(t\right)^{n}}+\varepsilon_{2}\left(x\right)\right],
\]
 where $\varepsilon_{2}\left(\nu t\right)=O\left(\nu^{-2}A\left(t\right)^{1-n}\widetilde{J}_{n-1}\left(\nu A\left(t\right)\right)\right)$
uniformly in $0\leqslant t\leqslant a$. 
\item \label{enu:2 est}(Airy regime) There exist functions $\eta_{0},$
$\xi_{1}\in C^{1}\left(0,\infty\right)$ such that, for any $b>0$
and for $x\geqslant b\nu$, (again with $t=x/\nu$), one has 
\[
\Lambda_{k}^{n-1}(x)=\left(-1\right)^{k}2^{1-n}r_{k}x^{\frac{n-1}{2}}\left[\frac{\eta_{0}\left(t\right)}{\nu^{1/3}}\mathrm{Ai}\left(\nu^{2/3}\Theta\left(t\right)\right)-\frac{\xi_{1}\left(t\right)}{\nu^{5/3}}\mathrm{Ai}'\left(\nu^{2/3}\Theta\left(t\right)\right)+\varepsilon_{2}\left(x\right)\right],
\]
where $\varepsilon_{2}\left(\nu t\right)=O\left(\nu^{-7/3}\left|\widetilde{\mathrm{Ai}}\left(\nu^{2/3}\Theta\left(t\right)\right)\right|\right)$
uniformly for $t\geqslant b$.
\end{enumerate}
We recall only the definitions of $\alpha_{0}$ and $\eta_{0}$, since
the explicit expressions for $\beta_{1}$ and $\xi_{1}$ are more
involved. For the latter we simply note their asymptotic behavior
as $t\to0^{+}$ and $t\to1$. We have
\begin{align}
\alpha_{0}\left(t\right)= & \left(1-t\right)^{-1/4}\left(\frac{A\left(t\right)}{\sqrt{t}}\right)^{n-1/2}=1+O\left(t\right), &  & t\to0^{+},\label{eq: def alpha 0}\\
\eta_{0}\left(t\right)= & t^{1/4-n/2}\left(\frac{4\Theta\left(t\right)}{t-1}\right)^{1/4}=2^{1/3}\left(1+O\left(1-t\right)\right), &  & t\to1,\label{eq:eta_0}\\
\beta_{1}\left(t\right)= & O\left(t\right) &  & t\to0^{+}\label{eq: beta}\\
\eta_{1}\left(t\right)= & O\left(1\right) &  & t\to1.\nonumber 
\end{align}
See \cite{FrenzenWong1988} for the details.

\section{\label{sec:Estimate I}Estimate for the term $\mathcal{I}$ in (\ref{eq:I x J})}

This section is devoted to estimate the term $\mathcal{I}$ in (\ref{eq:I x J})
from below. To this end, we first establish pointwise estimates for
$\widehat{\chi_{B_{\rho}}}$ and then derive lower bounds for its
averages. Finally, by restricting these bounds to the region $F_{\Lambda}$,
we obtain the desired estimates for $\mathcal{I}$.

\subsection{Pointwise estimates for $\widehat{\chi_{B}}$}

Observe that, since by (\ref{eq:Dilatazioni fourier})
\[
\widehat{\chi_{B_{\rho}}}\left(\lambda,k\right)=\rho^{2n+2}\widehat{\chi_{B}}\left(\lambda\rho^{2},k\right)
\]
where $B=B_{1}$, it suffices to estimate $\widehat{\chi_{B}}$. Furthermore,
since $\widehat{\chi_{B}}\left(\cdot,k\right)$ is even we can assume
$\lambda>0$. We obtain the following.
\begin{thm}
\label{thm:stime}There exists a function $\Omega\left(\lambda,\nu\right)\approx1$
defined on $\mathbb{R^{*}}\times\mathbb{N}$ such that
\begin{enumerate}
\item \label{enu:thm a}for every $0<\lambda\leqslant\nu$, we have 
\begin{align*}
\widehat{\chi_{B}}\left(\lambda,k\right) & =\frac{\sin\lambda}{\lambda}\left[\frac{J_{n}\left(\nu A\left(\frac{\lambda}{2\nu}\right)\right)}{\left(\lambda\nu\right)^{n/2}}\Omega\left(\lambda,\nu\right)+O\left(\frac{1}{\nu^{2}\left\langle \lambda\nu\right\rangle ^{n/2-1/4}}\right)\right];
\end{align*}
\item \label{enu:thm b}for every $\nu<\lambda\leqslant2\left(\nu+\nu^{1/3}\right)$,
we have
\[
\widehat{\chi_{B}}\left(\lambda,k\right)=\left(-1\right)^{k}\frac{\sin\lambda}{\lambda}\frac{1}{\left(\lambda\nu\right)^{n/2}}\left(\mathrm{IAi}\left(\nu^{2/3}\Theta\left(\frac{\lambda}{2\nu}\right)\right)\Omega\left(\lambda,\nu\right)+O\left(\frac{1}{\nu^{2/3}}\left\langle \nu^{2/3}\left(1-\frac{\lambda}{2\nu}\right)\right\rangle ^{-5/4}\right)\right);
\]
\item \label{enu:thm c}for every $2\left(\nu+\nu^{1/3}\right)<\lambda\leqslant3\nu$,
we have
\[
\widehat{\chi_{B}}\left(\lambda,k\right)=\left(-1\right)^{k}\frac{\sin\lambda}{\lambda^{n+1}}\left(\Omega\left(\lambda,\nu\right)+O\left(\frac{1}{\nu^{2/3}}\right)\right);
\]
\item \label{enu:thm d}there exists $c>0$ such that for every $\lambda>3\nu$,
we have
\[
\widehat{\chi_{B}}\left(\lambda,k\right)=\left(-1\right)^{k}\frac{\sin\lambda}{\lambda^{n+1}}2^{\frac{3n+1}{2}}\left(1+O\left(\nu^{-\frac{n-1}{2}}\lambda^{n/2-1}e^{-c\lambda}\right)\right).
\]
\end{enumerate}
\end{thm}

We split the proof of the four cases in the next paragraphs. Observe
that, since
\[
\chi_{B}\left(z,t\right)=\chi_{\left[-1,1\right]}\left(\left|z\right|\right)\chi_{\left[-1,1\right]}\left(t\right),
\]
we have
\[
\left(\chi_{B}\right)^{\lambda}\left(z\right)=\chi_{\left[-1,1\right]}\left(\left|z\right|\right)2\frac{\sin\lambda}{\lambda}.
\]
Therefore, by Theorem \ref{thm:Fourier transf radial}, we have
\begin{align}
\widehat{\chi_{B}}\left(\lambda,k\right)= & r_{k}2\frac{\sin\lambda}{\lambda}\int_{0}^{1}\left(\lambda r^{2}\right)^{\left(1-n\right)/2}\Lambda_{k}^{n-1}\left(\frac{1}{2}\lambda r^{2}\right)r^{2n-1}dr\label{eq: chi B}\\
= & r_{k}\frac{\sin\lambda}{\lambda^{n+1}}\left(2\nu\right)^{\left(n+1\right)/2}\int_{0}^{\frac{\lambda}{2\nu}}\Lambda_{k}^{n-1}\left(\nu t\right)t^{\left(n-1\right)/2}dt.\nonumber 
\end{align}

\subsubsection{Proof of (\ref{enu:thm a}) in Theorem \ref{thm:stime}.}

In the regime $0<\lambda\leqslant\nu$, if $t\in\left[0,\frac{\lambda}{2\nu}\right]$,
then $\nu t\leqslant\frac{\lambda}{2}\leqslant\frac{\nu}{2}$. Therefore,
by (\ref{enu:1 est}) in Theorem \ref{thm: FW} we get 
\begin{align*}
\widehat{\chi_{B}}\left(\lambda,k\right)= & r_{k}^{2}\frac{\sin\lambda}{\lambda^{n+1}}2^{\frac{3-n}{2}}\nu^{n}\left[\int_{0}^{\frac{\lambda}{2\nu}}J_{n-1}\left(\nu A\left(t\right)\right)\frac{t^{n-1}\alpha_{0}\left(t\right)}{A\left(t\right)^{n-1}}dt\right.\\
 & \qquad\qquad\qquad\left.-\frac{2}{\nu}\int_{0}^{\frac{\lambda}{2\nu}}J_{n}\left(\nu A\left(t\right)\right)\frac{t^{n-1}\beta_{1}\left(t\right)}{A\left(t\right)^{n}}dt+\int_{0}^{\frac{\lambda}{2\nu}}t^{n-1}\varepsilon_{2}\left(t\right)dt\right]\\
= & r_{k}^{2}\frac{\sin\lambda}{\lambda^{n+1}}2^{\frac{3-n}{2}}\nu^{n}\left[\textrm{I}_{1}-\textrm{I}_{2}+\textrm{I}_{3}\right].
\end{align*}

We start by estimating $\textrm{I}_{1}$. Recalling that $\frac{d}{du}\left(u^{n}J_{n}\left(u\right)\right)=u^{n}J_{n-1}\left(u\right)$,
and using the definition of $\alpha_{0}$ (\ref{eq: def alpha 0})
together with the derivative $A'$ (\ref{eq: A'}), we integrate by
parts and obtain
\begin{align*}
\textrm{I}_{1}= & \frac{1}{\nu^{n+1}}\int_{0}^{\frac{\lambda}{2\nu}}\left[\left(\nu A\left(t\right)\right)^{n}J_{n-1}\left(\nu A\left(t\right)\right)\right]\nu A'\left(t\right)\left[\frac{t^{n-1}\alpha_{0}\left(t\right)}{A'\left(t\right)A\left(t\right)^{2n-1}}\right]dt\\
= & \frac{2}{\nu^{n+1}}\int_{0}^{\frac{\lambda}{2\nu}}\frac{d}{dt}\left[\left(\nu A\left(t\right)\right)^{n}J_{n}\left(\nu A\left(t\right)\right)\right]\left[\frac{t^{n/2-1/4}}{\left(1-t\right)^{3/4}A\left(t\right)^{n-1/2}}\right]dt\\
= & \frac{2}{\nu}J_{n}\left(\nu A\left(\frac{\lambda}{2\nu}\right)\right)\frac{\left(\frac{\lambda}{2\nu}\right)^{n/2-1/4}A\left(\frac{\lambda}{2\nu}\right)^{1/2}}{\left(1-\frac{\lambda}{2\nu}\right)^{3/4}}\\
 & -\frac{2}{\nu^{n+1}}\int_{0}^{\frac{\lambda}{2\nu}}\left(\nu A\left(t\right)\right)^{n}J_{n}\left(\nu A\left(t\right)\right)\frac{d}{dt}\left[\frac{t^{n/2-1/4}}{\left(1-t\right)^{3/4}A\left(t\right)^{n-1/2}}\right]dt=\textrm{I}_{1,1}-\textrm{I}_{1,2}.
\end{align*}
Here the main term will be $\mathrm{I}_{1,1}$. 

The terms ${\rm I}_{1,2}$ and ${\rm I}_{2}$ have a similar behavior
and can be estimated together. Indeed,
\begin{align*}
\textrm{I}_{1,2}+\textrm{I}_{2} & =\frac{2}{\nu^{n+1}}\int_{0}^{\frac{\lambda}{2\nu}}\left(\nu A\left(t\right)\right)^{n}J_{n}\left(\nu A\left(t\right)\right)\left[\frac{d}{dt}\left(\frac{t^{n/2-1/4}}{\left(1-t\right)^{3/4}A\left(t\right)^{n-1/2}}\right)+\frac{t^{n-1}\beta_{1}\left(t\right)}{A\left(t\right)^{2n}}\right]dt\\
 & =\frac{2}{\nu^{n+3}}\int_{0}^{\frac{\lambda}{2\nu}}\nu A'\left(t\right)\left(\nu A\left(t\right)\right)^{n+1}J_{n}\left(\nu A\left(t\right)\right)F\left(t\right)dt,
\end{align*}
with
\[
F\left(t\right):=\frac{1}{A'\left(t\right)A\left(t\right)}\left[\frac{d}{dt}\left(\frac{t^{n/2-1/4}}{\left(1-t\right)^{3/4}A\left(t\right)^{n-1/2}}\right)+\frac{t^{n-1}\beta_{1}\left(t\right)}{A\left(t\right)^{2n}}\right].
\]
Since
\begin{align*}
A'\left(t\right)A\left(t\right) & =\frac{1}{4}\left(\left(\frac{1-t}{t}\right)^{1/2}\arcsin\left(\sqrt{t}\right)+1-t\right)\geqslant\frac{1}{4}\left(\left(1-t\right)^{1/2}+1-t\right)>\frac{1+\sqrt{2}}{8}, & t\in\left[0,\frac{1}{2}\right],
\end{align*}
the first factor is the inverse of a positive $C^{1}$ function, whereas
the latter factor is in $C^{1}\left(\left[0,\frac{1}{2}\right]\right)$
by the asymptotic expansions at $0^{+}$ of $A$ (\ref{eq:asympt A})
and $\beta_{1}$ (\ref{eq: beta}). Thus, $F\in C^{1}\left(\left[0,\frac{1}{2}\right]\right)$.
Integrating by parts, we have
\begin{align}
\textrm{I}_{1,2}+\textrm{I}_{2} & =\frac{2}{\nu^{2}}A\left(\frac{\lambda}{2\nu}\right)^{n+1}J_{n+1}\left(\nu A\left(\frac{\lambda}{2\nu}\right)\right)F\left(\frac{\lambda}{2\nu}\right)-\frac{2}{\nu^{2}}\int_{0}^{\frac{\lambda}{2\nu}}A\left(t\right)^{n+1}J_{n+1}\left(\nu A\left(t\right)\right)F'\left(t\right)dt\nonumber \\
 & =O\left(\frac{1}{\nu^{2}}\left(\frac{\lambda}{\nu}\right)^{n/2+1/2}\widetilde{J}_{n+1}\left(\nu A\left(\frac{\lambda}{2\nu}\right)\right)\right)+O\left(\frac{1}{\nu^{2}}\int_{0}^{\frac{\lambda}{2\nu}}t^{\frac{n+1}{2}}\widetilde{J}_{n+1}\left(\nu A\left(t\right)\right)dt\right)=\mathcal{O}_{1}+\mathcal{O}_{2},\label{eq:I1,2}
\end{align}
where we used the estimate for $A$ (\ref{eq:asympt A}), and the
definition of $\widetilde{J}_{n+1}$ (\ref{eq: def J tilde}). 

Finally, for $I_{3}$, we use the estimate for $\varepsilon_{2}$
provided by Theorem \ref{thm: FW} and (\ref{eq:asympt A}) to get
\begin{align}
\textrm{I}_{3} & =\int_{0}^{\frac{\lambda}{2\nu}}t^{n-1}\varepsilon_{2}\left(t\right)dt=O\left(\frac{1}{\nu^{2}}\int_{0}^{\frac{\lambda}{2\nu}}\widetilde{J}_{n-1}\left(\nu A\left(t\right)\right)\frac{t^{n-1}}{A\left(t\right)^{n-1}}dt\right)\nonumber \\
 & =O\left(\frac{1}{\nu^{2}}\int_{0}^{\frac{\lambda}{2\nu}}t^{\frac{n-1}{2}}\widetilde{J}_{n-1}\left(\nu A\left(t\right)\right)dt\right)=\mathcal{O}_{3}.\label{eq:I3}
\end{align}

We now show that $\mathcal{O}_{3}$ is larger than $\mathcal{O}_{1}$
and $\mathcal{O}_{2}$. Notice that, by (\ref{eq:J_tilde asympt}),
we have
\[
\tilde{J}_{n}\left(\nu A\left(t\right)\right)\approx\frac{\left(\nu A\left(t\right)\right)^{n}}{\left\langle \nu A\left(t\right)\right\rangle ^{n+1/2}}\approx\frac{\left(\nu t^{1/2}\right)^{n}}{\left\langle \nu t^{1/2}\right\rangle ^{n+1/2}}.
\]
A simple computation shows that 
\[
\mathcal{O}_{3}=O\left(\lambda^{n}\nu^{-3}\left\langle \lambda\nu\right\rangle ^{-n/2+1/4}\right).
\]
Similarly we obtain
\[
\mathcal{O}_{1}=O\left(\lambda^{n+1}\nu^{-2}\left\langle \lambda\nu\right\rangle ^{-n/2-3/4}\right)=O\left(\lambda^{n}\nu^{-3}\left\langle \lambda\nu\right\rangle ^{-n/2+1/4}\right).
\]
On the other hand, since
\[
t\tilde{J}_{n+1}\left(\nu A\left(t\right)\right)\lesssim\tilde{J}_{n-1}\left(\nu A\left(t\right)\right),
\]
we have $\mathcal{O}_{2}\lesssim\mathcal{O}_{3}$. 

We have proved that 
\begin{align*}
\textrm{I}_{1}-\textrm{I}_{2}+\textrm{I}_{3} & =\textrm{I}_{1,1}+O\left(\lambda^{n}\nu^{-3}\left\langle \lambda\nu\right\rangle ^{-n/2+1/4}\right)
\end{align*}
By $r_{k}^{2}\approx\nu^{1-n}$, the asymptotic expansion of $A$
(\ref{eq:asympt A}) and the fact that $t\in\left[0,\frac{1}{2}\right]$,
we finally get
\begin{align*}
\widehat{\chi_{B}}\left(\lambda,k\right) & =\frac{\sin\lambda}{\lambda}\left[\frac{J_{n}\left(\nu A\left(\frac{\lambda}{2\nu}\right)\right)}{\left(\lambda\nu\right)^{n/2}}\Omega_{1}\left(\lambda,\nu\right)+O\left(\nu^{-2}\left\langle \lambda\nu\right\rangle ^{-n/2+1/4}\right)\right],
\end{align*}
where
\[
\Omega\left(\lambda,\nu\right)=c\left[r_{k}^{2}\nu^{n-1}\right]\left(\frac{\lambda}{\nu}\right)^{-1/4}A\left(\frac{\lambda}{2\nu}\right)^{1/2}\left(1-\frac{\lambda}{2\nu}\right)^{-3/4}\approx1.
\]

\subsubsection{Proof of (\ref{enu:thm b}) in Theorem \ref{thm:stime}.}

The regime $\nu<\lambda\leqslant2\left(\nu+\nu^{1/3}\right)$ is covered
by a mix of estimates (\ref{enu:1 est}) and (\ref{enu:2 est}) in
Theorem \ref{thm: FW}. To smoothly connect the two estimates and
avoid boundary terms we introduce a cut-off function $\phi\in C^{\infty}\left(\mathbb{R}\right)$
such that $\phi\left(t\right)=0$ for $t\leqslant\frac{1}{3}$ and
$\phi\left(t\right)=1$ for $t\geqslant\frac{1}{2}$, and we write
\begin{equation}
\int_{0}^{\frac{\lambda}{2\nu}}\Lambda_{k}^{n-1}\left(\nu t\right)t^{\frac{n-1}{2}}dt=\int_{0}^{\infty}\Lambda_{k}^{n-1}\left(\nu t\right)t^{\frac{n-1}{2}}\left(1-\phi\left(t\right)\right)dt+\int_{0}^{\frac{\lambda}{2\nu}}\Lambda_{k}^{n-1}\left(\nu t\right)t^{\frac{n-1}{2}}\phi\left(t\right)dt.\label{eq: spezzo int}
\end{equation}
The first integral in the RHS can be controlled as in the previous
case. Exploiting the fact that $1-\phi$ vanishes for $t\geqslant\frac{1}{2}$
we obtain 
\begin{equation}
\int_{0}^{\infty}\Lambda_{k}^{n-1}\left(\nu t\right)t^{\frac{n-1}{2}}\left(1-\phi\left(t\right)\right)dt=O\left(\frac{1}{\nu^{5/2}}\right).\label{eq: int cutoff}
\end{equation}

To estimate the second integral we use (\ref{enu:2 est}) in Theorem
\ref{thm: FW} with $b=\frac{1}{3}$ to have 
\begin{align}
 & \int_{0}^{\frac{\lambda}{2\nu}}\Lambda_{k}^{n-1}\left(\nu t\right)t^{\frac{n-1}{2}}\phi\left(t\right)dt\nonumber \\
 & =\left(-1\right)^{k}r_{k}2^{1-n}\nu^{\frac{n-1}{2}}\int_{0}^{\frac{\lambda}{2\nu}}t^{n-1}\left[\frac{\eta_{0}\left(t\right)}{\nu^{1/3}}\mathrm{Ai}\left(\nu^{2/3}\Theta\left(t\right)\right)-\frac{\xi_{1}\left(t\right)}{\nu^{5/3}}\mathrm{Ai}'\left(\nu^{2/3}\Theta\left(t\right)\right)+\varepsilon_{2}\left(\nu t\right)\right]\phi\left(t\right)dt\label{eq: II1 II2 II3}\\
 & =\left(-1\right)^{k}r_{k}2^{1-n}\nu^{\frac{n-1}{2}}\left(\textrm{II}_{1}-\textrm{II}_{2}+\textrm{II}_{3}\right).\nonumber 
\end{align}
Since $\Theta'\left(t\right)>0$, integrating by parts and using the
definition of $\textrm{IAi}$ yields
\begin{align*}
\textrm{II}_{1}= & \frac{1}{\nu}\int_{0}^{\frac{\lambda}{2\nu}}\nu^{2/3}\Theta'\left(t\right)\mathrm{Ai}\left(\nu^{2/3}\Theta\left(t\right)\right)\frac{t^{n-1}\eta_{0}\left(t\right)\phi\left(t\right)}{\Theta'\left(t\right)}dt\\
= & \frac{1}{\nu}\mathrm{IAi}\left(\nu^{2/3}\Theta\left(\frac{\lambda}{2\nu}\right)\right)\frac{\left(\frac{\lambda}{2\nu}\right)^{n-1}\eta_{0}\left(\frac{\lambda}{2\nu}\right)}{\Theta'\left(\frac{\lambda}{2\nu}\right)}-\frac{1}{\nu}\int_{0}^{\frac{\lambda}{2\nu}}\mathrm{IAi}\left(\nu^{2/3}\Theta\left(t\right)\right)\frac{d}{dt}\left[\frac{t^{n-1}\eta_{0}\left(t\right)\phi\left(t\right)}{\Theta'\left(t\right)}\right]dt\\
= & \textrm{II}_{1,1}-\textrm{II}_{1,2}.
\end{align*}
The main term will be $\textrm{II}_{1,1}$. To estimate $\textrm{II}_{1,2}$,
we set 
\[
G_{1}\left(t\right):=\frac{1}{\Theta'\left(t\right)}\frac{d}{dt}\left[\frac{t^{n-1}\eta_{0}\left(t\right)\phi\left(t\right)}{\Theta'\left(t\right)}\right].
\]
Since $\Theta'>0$, $G_{1}$ is smooth. Integrating by parts gives
\begin{align*}
\textrm{II}_{1,2}= & \frac{1}{\nu^{5/3}}\int_{0}^{\frac{\lambda}{2\nu}}\nu^{2/3}\Theta'\left(t\right)\mathrm{IAi}\left(\nu^{2/3}\Theta\left(t\right)\right)G_{1}\left(t\right)dt\\
= & O\left(\frac{1}{\nu^{5/3}}\mathrm{IIAi}\left(\nu^{2/3}\Theta\left(\frac{\lambda}{2\nu}\right)\right)\right)+O\left(\frac{1}{\nu^{5/3}}\int_{0}^{\frac{\lambda}{2\nu}}\left|\mathrm{IIAi}\left(\nu^{2/3}\Theta\left(t\right)\right)\right|dt\right)
\end{align*}
In order to estimate the second error term, observe that, for $t\leqslant1-\nu^{-2/3}$,
we have 
\[
\nu^{2/3}\Theta\left(t\right)\leqslant\nu^{2/3}\Theta\left(1-\nu^{-2/3}\right)=-2^{-2/3}\left(1+O\left(\nu^{-2/3}\right)\right)\leqslant C<0,
\]
if $\nu$ large enough. By Lemma \ref{lem: stime IA} and (\ref{eq:That t->1}),
we have 
\[
\mathrm{IIAi}\left(\nu^{2/3}\Theta\left(t\right)\right)=O\left(\left\langle \nu^{2/3}\left(1-t\right)\right\rangle ^{-5/4}\right),
\]
and a straightforward computation gives
\[
\textrm{II}_{1,2}=O\left(\frac{1}{\nu^{5/3}}\left\langle \nu^{2/3}\left(1-\frac{\lambda}{2\nu}\right)\right\rangle ^{-5/4}\right).
\]
Let us consider $\textrm{II}_{2}$. Here we need the smooth function
\[
G_{2}\left(t\right):=\frac{t^{n-1}\xi_{1}\left(t\right)\phi\left(t\right)}{\Theta'\left(t\right)}.
\]
Integrating by parts and using Lemma \ref{lem: stime IA}, we obtain
\begin{align*}
\textrm{II}_{2} & =\frac{1}{\nu^{7/3}}\int_{0}^{\frac{\lambda}{2\nu}}\nu^{2/3}\Theta'\left(t\right)\mathrm{Ai}'\left(\nu^{2/3}\Theta\left(t\right)\right)G_{2}\left(t\right)dt\\
 & =O\left(\frac{1}{\nu^{7/3}}\mathrm{Ai}\left(\nu^{2/3}\Theta\left(\frac{\lambda}{2\nu}\right)\right)\right)+O\left(\frac{1}{\nu^{7/3}}\int_{0}^{\frac{\lambda}{2\nu}}\left|\mathrm{Ai}\left(\nu^{2/3}\Theta\left(t\right)\right)\right|dt\right)\\
 & =O\left(\frac{1}{\nu^{7/3}}\left\langle \nu^{2/3}\left(1-\frac{\lambda}{2\nu}\right)\right\rangle ^{-1/4}\right).
\end{align*}
Finally, for $\textrm{II}_{3}$, we have
\begin{align*}
\textrm{II}_{3} & =\int_{\frac{1}{3}}^{\frac{\lambda}{2\nu}}t^{n-1}\varepsilon_{2}\left(\nu t\right)\phi\left(t\right)dt=O\left(\frac{1}{\nu^{7/3}}\int_{\frac{1}{3}}^{\frac{\lambda}{2\nu}}t^{n-1}\left|\widetilde{\mathrm{Ai}}\left(\nu^{2/3}\Theta\left(t\right)\right)\right|dt\right)=O\left(\frac{1}{\nu^{5/2}}\right).
\end{align*}
Since both $\textrm{II}_{2}$ and $\textrm{II}_{3}$ are $O\left(\frac{1}{\nu^{5/3}}\left\langle \nu^{2/3}\left(1-\frac{\lambda}{2\nu}\right)\right\rangle ^{-5/4}\right)$,
we obtain 
\begin{align*}
\textrm{II}_{1}-\textrm{II}_{2}+\textrm{II}_{3} & =\textrm{II}_{1,1}+O\left(\frac{1}{\nu^{5/3}}\left\langle \nu^{2/3}\left(1-\frac{\lambda}{2\nu}\right)\right\rangle ^{-5/4}\right).
\end{align*}
Using (\ref{eq: spezzo int}), (\ref{eq: int cutoff}), and (\ref{eq: II1 II2 II3}),
we have
\begin{align}
\int_{0}^{\frac{\lambda}{2\nu}}\Lambda_{k}^{n-1}\left(\nu t\right)t^{\frac{n-1}{2}}dt= & \left(-1\right)^{k}r_{k}2^{1-n}\nu^{\frac{n-1}{2}}\left(\frac{1}{\nu}\mathrm{IAi}\left(\nu^{2/3}\Theta\left(\frac{\lambda}{2\nu}\right)\right)\frac{\left(\frac{\lambda}{2\nu}\right)^{n-1}\eta_{0}\left(\frac{\lambda}{2\nu}\right)}{\Theta'\left(\frac{\lambda}{2\nu}\right)}\right.\label{eq: int Lam reg 2}\\
 & \left.+O\left(\frac{1}{\nu^{5/3}}\left\langle \nu^{2/3}\left(1-\frac{\lambda}{2\nu}\right)\right\rangle ^{-5/4}\right)\right)\nonumber 
\end{align}
and finally, by (\ref{eq: chi B}), 
\[
\widehat{\chi_{B}}\left(\lambda,k\right)=\left(-1\right)^{k}\frac{\sin\lambda}{\lambda}\frac{1}{\left(\lambda\nu\right)^{n/2}}\left[\mathrm{IAi}\left(\nu^{2/3}\Theta\left(\frac{\lambda}{2\nu}\right)\right)\Omega\left(\lambda,\nu\right)+O\left(\frac{1}{\nu^{2/3}}\left\langle \nu^{2/3}\left(1-\frac{\lambda}{2\nu}\right)\right\rangle ^{-5/4}\right)\right]
\]
where we set
\[
\Omega\left(\lambda,\nu\right)=\left[r_{k}^{2}\nu^{n-1}\right]2^{3/2-n}\frac{\left(\frac{\lambda}{2\nu}\right)^{n/2-1}\eta_{0}\left(\frac{\lambda}{2\nu}\right)}{\Theta'\left(\frac{\lambda}{2\nu}\right)}\approx1.
\]

\subsubsection{Proof of (\ref{enu:thm c}) in Theorem \ref{thm:stime}.}

Here we consider the case $2\left(\nu+\nu^{1/3}\right)\leqslant\lambda\leqslant3\nu$.
We have 
\begin{align*}
\left(-1\right)^{k}\int_{0}^{\frac{\lambda}{2\nu}}\Lambda_{k}^{n-1}\left(t\nu\right)t^{\frac{n-1}{2}}dt= & \left(-1\right)^{k}\int_{0}^{1+\nu^{-2/3}}\Lambda_{k}^{n-1}\left(t\nu\right)t^{\frac{n-1}{2}}dt+\left(-1\right)^{k}\int_{1+\nu^{-2/3}}^{\frac{\lambda}{2\nu}}\Lambda_{k}^{n-1}\left(t\nu\right)t^{\frac{n-1}{2}}dt\\
= & \text{III}_{1}+\text{III}{}_{2}.
\end{align*}
We estimate $\text{III}_{1}$ using (\ref{eq: int Lam reg 2}). We
have 
\begin{align*}
\text{III}_{1} & =\frac{1}{\nu}\Omega\left(\nu\right)+O\left(\frac{1}{\nu^{5/3}}\right),
\end{align*}
where we set 
\[
\Omega\left(\nu\right)=r_{k}2^{1-n}\nu^{\frac{n-1}{2}}\mathrm{IAi}\left(\nu^{2/3}\Theta\left(1+\nu^{-2/3}\right)\right)\frac{\left(1+\nu^{-2/3}\right)^{n-1}\eta_{0}\left(1+\nu^{-2/3}\right)}{\Theta'\left(1+\nu^{-2/3}\right)}\approx1.
\]
To estimate $\text{III}_{2}$ it is enough the first order expansion
Laguerre polynomials in the Airy regime. By (5.13) and (5.19) in \cite{FrenzenWong1988}
we have
\[
\Lambda_{k}^{n-1}(x)=\left(-1\right)^{k}2^{1-n}r_{k}x^{\left(n-1\right)/2}\left[\frac{\eta_{0}\left(t\right)}{\nu^{1/3}}\mathrm{Ai}\left(\nu^{2/3}\Theta\left(t\right)\right)+\varepsilon_{1}\left(x\right)\right],
\]
where, by (\ref{eq:Stime Airy}), we rewrite the estimate for $\varepsilon_{1}$
in this region as
\[
\varepsilon_{1}\left(t\nu\right)=O\left(\frac{1}{\nu^{5/3}}\left|{\rm Ai}'\left(\nu^{2/3}\Theta\left(t\right)\right)\right|\right)=O\left(\frac{1}{\nu^{5/3}}\left(\nu^{2/3}\Theta\left(t\right)\right)^{1/4}e^{-\frac{2}{3}\left(\nu^{2/3}\Theta\left(t\right)\right)^{3/2}}\right).
\]
Hence
\[
\text{III}_{2}=2^{1-n}r_{k}\nu^{\frac{n-1}{2}}\int_{1+\nu^{-2/3}}^{\frac{\lambda}{2\nu}}\left[\frac{\eta_{0}\left(t\right)}{\nu^{1/3}}\mathrm{Ai}\left(\nu^{2/3}\Theta\left(t\right)\right)+\varepsilon_{1}\left(t\nu\right)\right]t^{n-1}dt=2^{1-n}r_{k}\nu^{\frac{n-1}{2}}\left(\text{III}_{2,1}+\text{III}_{2,2}\right).
\]
Observe that
\begin{align*}
\text{III}_{2,2} & =O\left(\int_{1+\nu^{-2/3}}^{\frac{\lambda}{2\nu}}\frac{1}{\nu^{5/3}}\left(\nu^{2/3}\Theta\left(t\right)\right)^{1/4}e^{-\frac{2}{3}\left(\nu^{2/3}\Theta\left(t\right)\right)^{3/2}}dt\right)\\
 & =O\left(\frac{1}{\nu^{7/3}}\int_{0}^{+\infty}u^{1/4}e^{-\frac{2}{3}u^{3/2}}du\right)=O\left(\frac{1}{\nu^{7/3}}\right),
\end{align*}
and therefore 
\begin{align*}
\text{III}_{1}+\text{III}_{2}= & \frac{1}{\nu}\widetilde{\Omega}\left(\lambda,\nu\right)+O\left(\frac{1}{\nu^{5/3}}\right)
\end{align*}
where we set $\widetilde{\Omega}\left(\lambda,\nu\right)=\Omega\left(\nu\right)+2^{1-n}r_{k}\nu^{\frac{n-1}{2}}\nu\text{III}_{2,1}$.
Since $\mathrm{Ai}\left(\nu^{2/3}\Theta\left(t\right)\right)>0$ we
have $\text{III}_{2,1}>0$ so that $c_{1}\leqslant\Omega\left(\nu\right)\leqslant\widetilde{\Omega}\left(\lambda,\nu\right)$.
Also, by (\ref{eq: rk}) and Lemma \ref{lem: A beta alpha},
\begin{align*}
\widetilde{\Omega}\left(\lambda,\nu\right)= & \Omega\left(\nu\right)+2^{1-n}r_{k}\nu^{\frac{n-1}{2}}\int_{1+\nu^{-2/3}}^{\frac{\lambda}{2\nu}}\nu^{2/3}\Theta'\left(t\right)\mathrm{Ai}\left(\nu^{2/3}\Theta\left(t\right)\right)\frac{\eta_{0}\left(t\right)t^{n-1}}{\Theta'\left(t\right)}dt\\
\lesssim & 1+\int_{1+\nu^{-2/3}}^{\frac{\lambda}{2\nu}}\nu^{2/3}\Theta'\left(t\right)\mathrm{Ai}\left(\nu^{2/3}\Theta\left(t\right)\right)dt\leqslant c_{2}.
\end{align*}
So that $\widetilde{\Omega}\left(\lambda,\nu\right)\approx1$. Collecting
the above estimates and using (\ref{eq: chi B}) gives 
\begin{align*}
\widehat{\chi_{B}}\left(\lambda,k\right) & =r_{k}\frac{\sin\lambda}{\lambda^{n+1}}\left(2\nu\right)^{\frac{n+1}{2}}\int_{0}^{\frac{\lambda}{2\nu}}\Lambda_{k}^{n-1}\left(\nu t\right)t^{\frac{n-1}{2}}dt\\
 & =r_{k}\frac{\sin\lambda}{\lambda^{n+1}}\left(2\nu\right)^{\frac{n+1}{2}}\left(-1\right)^{k}\left(\frac{1}{\nu}\tilde{\Omega}\left(\lambda,\nu\right)+O\left(\frac{1}{\nu^{5/3}}\right)\right)\\
 & =\left(-1\right)^{k}\frac{\sin\lambda}{\lambda^{n+1}}\left(\Omega\left(\lambda,\nu\right)+O\left(\frac{1}{\nu^{2/3}}\right)\right)
\end{align*}
for a suitable function $\Omega\left(\lambda,\nu\right)\approx1$. 

\subsubsection{Proof of (\ref{enu:thm d}) in Theorem \ref{thm:stime}.}

For this last estimate, we use a different technique. Since we know
the value of the integral in $\left(0,+\infty\right)$ we proceed
by subtraction,
\begin{equation}
\int_{0}^{\frac{\lambda}{2\nu}}\Lambda_{k}^{n-1}\left(\nu t\right)t^{\frac{n-1}{2}}dt=\int_{0}^{+\infty}\Lambda_{k}^{n-1}\left(\nu t\right)t^{\frac{n-1}{2}}dt-\int_{\frac{\lambda}{2\nu}}^{+\infty}\Lambda_{k}^{n-1}\left(\nu t\right)t^{\frac{n-1}{2}}dt.\label{eq:different technique}
\end{equation}
We will show that the first term is the main term. By \cite[18.17.34]{NIST},
we have
\[
\int_{0}^{\infty}L_{k}^{n-1}\left(x\right)e^{-\frac{1}{2}x}x^{n-1}dx=\left(-1\right)^{k}\frac{2^{n}}{r_{k}^{2}}.
\]
Hence,
\begin{align*}
\int_{0}^{+\infty}\Lambda_{k}^{n-1}\left(\nu t\right)t^{\frac{n-1}{2}}dt & =\frac{r_{k}}{\nu^{\frac{n+1}{2}}}\int_{0}^{+\infty}L_{k}^{n-1}\left(x\right)e^{-\frac{x}{2}}x^{n-1}dx=\left(-1\right)^{k}2^{n}r_{k}^{-1}\nu^{-\frac{n+1}{2}}.
\end{align*}
It remains to estimate the second integral in the RHS of (\ref{eq:different technique}).
Since $\nu^{2/3}\Theta\left(t\right)\geqslant C>0$, using \ref{enu:2 est}
of Theorem (\ref{thm: FW}) and (\ref{eq:Stime Airy}) we have
\begin{align*}
\int_{\frac{\lambda}{2\nu}}^{\infty}\Lambda_{k}^{\left(n-1\right)}\left(\nu t\right)t^{\frac{n-1}{2}}dt & =O\left(\frac{1}{\nu^{1/3}}\int_{\frac{\lambda}{2\nu}}^{\infty}t^{n-1}\eta_{0}\left(t\right)\mathrm{Ai}\left(\nu^{2/3}\Theta\left(t\right)\right)dt\right)\\
 & =O\left(\frac{1}{\nu^{1/3}}\int_{\frac{\lambda}{2\nu}}^{\infty}t^{n-1}\eta_{0}\left(t\right)\left(\nu^{2/3}\Theta\left(t\right)\right)^{-1/4}e^{-\frac{2}{3}\left(\nu^{2/3}\Theta\left(t\right)\right)^{3/2}}dt\right).
\end{align*}
Using (\ref{eq:eta_0}) and (\ref{eq: theta grande}) we obtain
\begin{align*}
\int_{\frac{\lambda}{2\nu}}^{\infty}\Lambda_{k}^{\left(n-1\right)}\left(\nu t\right)t^{\frac{n-1}{2}}dt= & O\left(\frac{1}{\nu^{1/2}}\int_{\frac{\lambda}{2\nu}}^{\infty}t^{n/2-1}e^{-c\,\nu t}dt\right)\\
= & O\left(\nu^{-\frac{n+1}{2}}\int_{c'\lambda}^{\infty}u^{n/2-1}e^{-u}du\right)=O\left(\nu^{-\frac{n+1}{2}}\lambda^{n/2-1}e^{-c'\lambda}\right).
\end{align*}
Hence
\begin{align*}
\widehat{\chi_{B}}\left(\lambda,k\right) & =r_{k}\frac{\sin\lambda}{\lambda^{n+1}}\left(2\nu\right)^{\frac{n+1}{2}}\left(\left(-1\right)^{k}2^{n}\nu^{-\frac{n+1}{2}}r_{k}^{-1}+O\left(\nu^{-\frac{n+1}{2}}\lambda^{n/2-1}e^{-c'\lambda}\right)\right)\\
 & =\left(-1\right)^{k}\frac{\sin\lambda}{\lambda^{n+1}}2^{\frac{3n+1}{2}}\left(1+O\left(\nu^{-\frac{n-1}{2}}\lambda^{n/2-1}e^{-c'\lambda}\right)\right).
\end{align*}

\subsection{Estimates from below of the square average of $\widehat{\chi_{B}}$}

The pointwise estimates for $\widehat{\chi_{B}}$ proved in the previous
section exhibit infinitely many zeros, and therefore no direct lower
bound is available. To overcome this difficulty, we average with respect
to the radius $\rho$ of the ball and estimate 
\begin{equation}
\int_{0}^{1}\left|\widehat{\chi_{B_{\rho}}}\left(\lambda,k\right)\right|^{2}d\rho=\int_{0}^{1}\rho^{2Q}\left|\widehat{\chi_{B}}\left(\rho^{2}\lambda,k\right)\right|^{2}d\rho,\label{eq:dilataz}
\end{equation}
where the identity follows from the Fourier dilation formula (\ref{eq:Dilatazioni fourier}).
This averaging suppresses oscillatory effects, so that only the decay
properties of the underlying Bessel and Airy terms remain, which can
then be controlled from below. We collect the resulting estimates
below.
\begin{prop}
\label{prop: stima media-1}There exists $\nu_{0}\in\mathbb{N}$ such
that for $\nu\geqslant\nu_{0}$, we have
\[
\int_{0}^{1}\left|\widehat{\chi_{B_{\rho}}}\left(\lambda,k\right)\right|^{2}d\rho\gtrsim\begin{cases}
\left\langle \lambda\nu\right\rangle ^{-Q/2+1/2}\left\langle \lambda\right\rangle ^{-2}, & \text{if }0<\lambda\leqslant\nu,\\
\lambda^{-Q-2/3}\left\langle \nu^{2/3}\left(1-\frac{\lambda}{2\nu}\right)\right\rangle ^{-1/2}, & \text{if }\nu<\lambda\leqslant2\left(\nu+\nu^{1/3}\right),\\
\lambda^{-Q-1}\left(\lambda-2\nu\right), & \text{if }2\left(\nu+\nu^{1/3}\right)<\lambda.
\end{cases}
\]
\end{prop}

We split the proof of the above theorem in several lemmas.
\begin{lem}
\label{lem:lambda<1/nu}For $\nu$ large enough and $0<\lambda<\nu$,
we have
\begin{equation}
\int_{0}^{1}\left|\widehat{\chi_{B_{\rho}}}\left(\lambda,k\right)\right|^{2}d\rho\gtrsim\left\langle \lambda\nu\right\rangle ^{-Q/2+1/2}\left\langle \lambda\right\rangle ^{-2}.\label{eq:Lemma 20}
\end{equation}
\end{lem}

\begin{proof}
We consider three different cases: $0<\lambda\leqslant\nu^{-1}$,
$\nu^{-1}<\lambda\leqslant1$, and $1<\lambda\leqslant\nu$. In the
first case $\widehat{\chi_{B_{\rho}}}\left(\lambda,k\right)$ does
not oscillate, in the second the oscillations are only dues to the
Bessel function, in the last case both the Bessel and the sine factor
oscillate. Assume $0<\lambda\leqslant\nu^{-1}$. Using (\ref{eq:Bound A(t)}),
we have $\nu A\left(\frac{\lambda}{2\nu}\right)\leqslant2^{-1/2}$,
so that, by (\ref{eq: bessel picc}), 
\[
J_{n}\left(\nu A\left(\frac{\lambda}{2\nu}\right)\right)\approx\left(\nu A\left(\frac{\lambda}{2\nu}\right)\right)^{n}\approx\left(\nu\lambda\right)^{n/2}
\]
Hence, using (\ref{eq:dilataz}) and (\ref{enu:thm a}) of Theorem
\ref{thm:stime}, if $\nu$ is large enough, we get
\begin{align*}
\int_{0}^{1}\left|\widehat{\chi_{B_{\rho}}}\left(\lambda,k\right)\right|^{2}d\rho\gtrsim & \int_{0}^{1}\rho^{2Q}\left|\frac{\sin\rho^{2}\lambda}{\rho^{2}\lambda}\frac{1}{\left(\rho^{2}\lambda\nu\right)^{n/2}}J_{n}\left(\nu A\left(\frac{\rho^{2}\lambda}{2\nu}\right)\right)\Omega\left(\rho^{2}\lambda,\nu\right)\right|^{2}d\rho+O\left(\frac{1}{\nu^{4}}\int_{0}^{1}\rho^{2Q}d\rho\right)\\
\gtrsim & \int_{0}^{1}\rho^{2Q}d\rho+O\left(\frac{1}{\nu^{4}}\int_{0}^{1}\rho^{2Q}d\rho\right)\gtrsim1\approx\left\langle \lambda\nu\right\rangle ^{-Q/2+1/2}\left\langle \lambda\right\rangle ^{-2}.
\end{align*}
Assume now $\nu^{-1}\leqslant\lambda<1$, then, letting
\[
\Psi\left(\lambda,\nu\right)=\left(\lambda\nu\right)^{1/4}\frac{\sin\lambda}{\lambda}J_{n}\left(\nu A\left(\frac{\lambda}{2\nu}\right)\right)\Omega\left(\lambda,\nu\right),
\]
from (\ref{enu:thm a}) in Theorem \ref{thm:stime}, we obtain 
\begin{align*}
\widehat{\chi_{B}}\left(\lambda,k\right) & =\left(\lambda\nu\right)^{-n/2-1/4}\left[\Psi\left(\lambda,\nu\right)+O\left(\frac{1}{\nu}\left(\frac{\lambda}{\nu}\right)^{1/2}\right)\right].
\end{align*}
Hence
\begin{align*}
\int_{0}^{1}\left|\widehat{\chi_{B_{\rho}}}\left(\lambda,k\right)\right|^{2}d\rho & \gtrsim\left(\lambda\nu\right)^{-n-1/2}\left(\int_{\frac{1}{2}}^{1}\left|\Psi\left(\rho^{2}\lambda,\nu\right)\right|^{2}d\rho+O\left(\frac{1}{\nu^{3}}\right)\right).
\end{align*}
Using (\ref{eq: A'}), (\ref{eq:Bound A(t)}), and Lemma \ref{lem:Bessel^2},
we obtain 
\begin{align}
\int_{\frac{1}{2}}^{1}\left|\Psi\left(\rho^{2}\lambda,\nu\right)\right|^{2}d\rho\approx & \frac{1}{\lambda}\int_{\frac{\lambda}{4}}^{\lambda}\left|\left(u\nu\right)^{1/4}\frac{\sin u}{u}J_{n}\left(\nu A\left(\frac{u}{2\nu}\right)\right)\Omega\left(u,\nu\right)\right|^{2}du\nonumber \\
\gtrsim & \frac{1}{\lambda}\int_{\frac{\lambda}{4}}^{\lambda}\left|\left(\nu A\left(\frac{u}{2\nu}\right)\right)^{1/2}J_{n}\left(\nu A\left(\frac{u}{2\nu}\right)\right)\right|^{2}du\label{eq: stima int J_n}\\
\gtrsim & \frac{1}{\lambda}\left(\frac{\lambda}{\nu}\right)^{1/2}\int_{\frac{1}{2}\sqrt{\frac{\lambda\nu}{2}}}^{\frac{\pi}{4}\sqrt{\frac{\lambda\nu}{2}}}\left|x^{1/2}J_{n}\left(x\right)\right|^{2}dx\gtrsim1.\nonumber 
\end{align}
and (\ref{eq:Lemma 20}) follows. Let now $1<\lambda\leqslant\nu$,
and let 
\[
\Psi\left(\lambda,\nu\right)=\left(\lambda\nu\right)^{1/4}\sin\lambda J_{n}\left(\nu A\left(\frac{\lambda}{2\nu}\right)\right)\Omega\left(\lambda,\nu\right),
\]
then by (\ref{enu:thm a}) in Theorem \ref{thm:stime} we have
\begin{align*}
\widehat{\chi_{B}}\left(\lambda,k\right)=\lambda^{-1}\left(\lambda\nu\right)^{-n/2-1/4}\left[\Psi\left(\lambda,\nu\right)+O\left(\frac{1}{\nu}\left(\frac{\lambda}{\nu}\right)^{1/2}\right)\right].
\end{align*}
Hence
\begin{align*}
\int_{0}^{1}\left|\widehat{\chi_{B_{\rho}}}\left(\lambda,k\right)\right|^{2}d\rho\geqslant & \int_{\frac{1}{\sqrt{2}}}^{1}\rho^{2Q}\left|\left(\rho^{2}\lambda\right)^{-1}\left(\rho^{2}\lambda\nu\right)^{-n/2-1/4}\left[\Psi\left(\rho^{2}\lambda,\nu\right)+O\left(\frac{1}{\nu}\left(\frac{\rho^{2}\lambda}{\nu}\right)^{1/2}\right)\right]\right|^{2}d\rho\\
\gtrsim & \lambda^{-2}\left(\lambda\nu\right)^{-n-1/2}\left[\int_{\frac{1}{\sqrt{2}}}^{1}\left|\Psi\left(\rho^{2}\lambda,\nu\right)\right|^{2}d\rho+O\left(\frac{1}{\nu^{2}}\right)\right].
\end{align*}
Using $\Omega\approx1$ and (\ref{eq:Bound A(t)}), we see that
\begin{align*}
\int_{\frac{1}{\sqrt{2}}}^{1}\left|\Psi\left(\rho^{2}\lambda,\nu\right)\right|^{2}d\rho\approx & \frac{1}{\lambda}\int_{\frac{\lambda}{2}}^{\lambda}\left|\Psi\left(u,\nu\right)\right|^{2}du\gtrsim\frac{1}{\lambda}\int_{\frac{\lambda}{2}}^{\lambda}\left|\sin\left(u\right)\left(\nu A\left(\frac{u}{2\nu}\right)\right)^{1/2}J_{n}\left(\nu A\left(\frac{u}{2\nu}\right)\right)\right|^{2}du.
\end{align*}
Let $\varepsilon>0$ and let $U_{\varepsilon}=\left\{ t\in\mathbb{R}\colon d\left(t,\pi\mathbb{Z}\right)>\varepsilon\right\} $.
Since $\left|\sin\left(u\right)\right|>c_{\varepsilon}$ if $u\in U_{\varepsilon}$,
and $t^{1/2}J_{n}\left(t\right)\lesssim1$, we have
\begin{align*}
\int_{\frac{1}{\sqrt{2}}}^{1}\left|\Psi\left(\rho^{2}\lambda,\nu\right)\right|^{2}d\rho\gtrsim & \frac{1}{\lambda}\left[\int_{\frac{\lambda}{2}}^{\lambda}\left|\left(\nu A\left(\frac{u}{2\nu}\right)\right)^{1/2}J_{n}\left(\nu A\left(\frac{u}{2\nu}\right)\right)\right|^{2}du+O\left(\left|\left[\frac{\lambda}{2},\lambda\right]\setminus U_{\varepsilon}\right|\right)\right]\\
= & \frac{1}{\lambda}\int_{\frac{\lambda}{2}}^{\lambda}\left|\left(\nu A\left(\frac{u}{2\nu}\right)\right)^{1/2}J_{n}\left(\nu A\left(\frac{u}{2\nu}\right)\right)\right|^{2}du+O\left(\varepsilon\right).
\end{align*}
For $\varepsilon$ small enough, the same computation as in (\ref{eq: stima int J_n})
gives 
\[
\int_{\frac{1}{\sqrt{2}}}^{1}\left|\Psi\left(\rho^{2}\lambda,\nu\right)\right|^{2}d\rho\gtrsim1+O\left(\varepsilon\right)>c.
\]
\end{proof}
\begin{lem}
Let $\nu<\lambda\leqslant2\left(\nu+2\nu^{1/3}\right)$ with $\nu$
large enough, then
\begin{equation}
\int_{0}^{1}\left|\widehat{\chi_{B_{\rho}}}\left(\lambda,k\right)\right|^{2}d\rho\gtrsim\lambda^{-Q-2/3}\left\langle \nu^{2/3}\left(1-\frac{\lambda}{2\nu}\right)\right\rangle ^{-1/2}.\label{eq: Lemma 21}
\end{equation}
\end{lem}

\begin{proof}
Assume first $\nu<\lambda\leqslant\frac{3}{2}\nu$. By the previous
Lemma, we have 
\begin{align*}
\int_{0}^{1}\left|\widehat{\chi_{B_{\rho}}}\left(\lambda,k\right)\right|^{2}d\rho\geqslant & \int_{0}^{\sqrt{\frac{\nu}{\lambda}}}\rho^{2Q}\left|\widehat{\chi_{B}}\left(\rho^{2}\lambda,k\right)\right|^{2}d\rho\\
= & \left(\frac{\nu}{\lambda}\right)^{Q+1/2}\int_{0}^{1}\rho^{2Q}\left|\widehat{\chi_{B}}\left(\nu\rho^{2},k\right)\right|^{2}d\rho\\
\gtrsim & \nu^{-Q-1}\approx\lambda^{-Q-2/3}\left\langle \nu^{2/3}\left(1-\frac{\lambda}{2\nu}\right)\right\rangle ^{-1/2}.
\end{align*}
Let now $\frac{3}{2}\nu\leqslant\lambda\leqslant2\left(\nu-\nu^{1/3}\right)$,
by (\ref{enu:thm b}) in Theorem \ref{thm:stime}, we have
\[
\widehat{\chi_{B}}\left(\lambda,k\right)=\frac{1}{\lambda}\frac{1}{\left(\lambda\nu\right)^{n/2}}\left(\nu^{2/3}\left(1-\frac{\lambda}{2\nu}\right)\right)^{-3/4}\left[\Psi\left(\lambda,\nu\right)+O\left(\nu^{-1}\left(1-\frac{\lambda}{2\nu}\right)^{-1/2}\right)\right]
\]
where we set
\begin{align*}
\Psi\left(\lambda,\nu\right)= & \left(-1\right)^{k}\sin\left(\lambda\right)\Omega\left(\lambda,\nu\right)\left(\nu^{2/3}\left(1-\frac{\lambda}{2\nu}\right)\right)^{3/4}\mathrm{IAi}\left(\nu^{2/3}\Theta\left(\frac{\lambda}{2\nu}\right)\right)
\end{align*}
Hence
\begin{align*}
\int_{0}^{1}\left|\widehat{\chi_{B_{\rho}}}\left(\lambda,k\right)\right|^{2}d\rho\gtrsim & \lambda^{-Q-1}\left[\int_{\sqrt{\frac{\nu}{\lambda}}}^{1}\left(1-\frac{\rho^{2}\lambda}{2\nu}\right)^{-3/2}\left|\Psi\left(\rho^{2}\lambda,\nu\right)\right|^{2}d\rho+O\left(\frac{1}{\nu^{2}}\int_{\sqrt{\frac{\nu}{\lambda}}}^{1}\left(1-\frac{\rho^{2}\lambda}{2\nu}\right)^{-5/2}d\rho\right)\right]\\
= & \lambda^{-Q-1}\int_{\sqrt{\frac{\nu}{\lambda}}}^{1}\left(1-\frac{\rho^{2}\lambda}{2\nu}\right)^{-3/2}\left|\Psi\left(\rho^{2}\lambda,\nu\right)\right|^{2}d\rho+O\left(\lambda^{-Q-2}\right).
\end{align*}
It remains to check that 
\[
\int_{\sqrt{\frac{\nu}{\lambda}}}^{1}\left(1-\frac{\rho^{2}\lambda}{2\nu}\right)^{-3/2}\left|\Psi\left(\rho^{2}\lambda,\nu\right)\right|^{2}d\rho\gtrsim\left(1-\frac{\lambda}{2\nu}\right)^{-1/2}.
\]
Let $\varepsilon>0$ and let $U_{\varepsilon}=\left\{ u\in\mathbb{R}\colon d\left(u,\pi\mathbb{Z}\right)>\varepsilon\right\} $,
then
\begin{align*}
 & \int_{\sqrt{\frac{\nu}{\lambda}}}^{1}\left(1-\frac{\rho^{2}\lambda}{2\nu}\right)^{-3/2}\left|\Psi\left(\rho^{2}\lambda,\nu\right)\right|^{2}d\rho\gtrsim\nu\int_{\sqrt{\frac{\nu}{\lambda}}}^{1}\left|\sin\left(\rho^{2}\lambda\right)\mathrm{IAi}\left(\nu^{2/3}\Theta\left(\frac{\rho^{2}\lambda}{2\nu}\right)\right)\right|^{2}d\rho\\
\gtrsim & \int_{\nu}^{\lambda}\left|\sin u\right|^{2}\left|\mathrm{IAi}\left(\nu^{2/3}\Theta\left(\frac{u}{2\nu}\right)\right)\right|^{2}du\\
\gtrsim & \int_{\nu}^{\lambda}\left|\mathrm{IAi}\left(\nu^{2/3}\Theta\left(\frac{u}{2\nu}\right)\right)\right|^{2}du-\int_{\left[\nu,\lambda\right]\cap U_{\varepsilon}^{c}}\left|\mathrm{IAi}\left(\nu^{2/3}\Theta\left(\frac{u}{2\nu}\right)\right)\right|^{2}du=S_{1}-S_{2}.
\end{align*}
First notice that, by Lemma \ref{lem: stime IA} and Lemma \ref{lem: A beta alpha},
we have that
\begin{align*}
S_{2} & \lesssim\int_{\left[\nu,\lambda\right]\cap U_{\varepsilon}^{c}}\left(-\nu^{2/3}\Theta\left(\frac{u}{2\nu}\right)\right)^{-3/2}du\lesssim\frac{1}{\nu}\int_{\left[\nu,\lambda\right]\cap U_{\varepsilon}^{c}}\left(1-\frac{u}{2\nu}\right)^{-3/2}du\lesssim\int_{\left[\frac{1}{2},\frac{\lambda}{2\nu}\right]\cap\frac{U_{\varepsilon}^{c}}{2\nu}}\left(1-x\right)^{-3/2}dx.
\end{align*}
Decomposing the domain of integration in the last integral as
\[
\left[\frac{1}{2},\frac{\lambda}{2\nu}\right]\cap\frac{U_{\varepsilon}^{c}}{2\nu}\subseteq\bigsqcup_{k=\left\lfloor \frac{\nu}{\pi}\right\rfloor }^{\left\lfloor \frac{\lambda+\varepsilon}{\pi}\right\rfloor }\left(\frac{\pi k}{2\nu}-\frac{\varepsilon}{2\nu},\frac{\pi k}{2\nu}+\frac{\varepsilon}{2\nu}\right),
\]
it is not difficult to check that $S_{2}\lesssim\varepsilon\left(1-\frac{\lambda}{2\nu}\right)^{-1/2}$.
By Lemma \ref{lem: A beta alpha},
\[
\frac{d}{du}\left[\nu^{2/3}\Theta\left(\frac{u}{2\nu}\right)\right]=\frac{1}{2}\nu^{-1/3}\Theta'\left(\frac{u}{2\nu}\right)\approx\nu^{-1/3},
\]
so that, by Lemma \ref{lem:Bessel^2}, we have
\begin{align*}
S_{1}=\int_{\nu}^{\lambda}\left|\mathrm{IAi}\left(\nu^{2/3}\Theta\left(\frac{u}{2\nu}\right)\right)\right|^{2}du & \approx\nu^{1/3}\int_{\nu^{2/3}\Theta\left(\frac{1}{2}\right)}^{\nu^{2/3}\Theta\left(\frac{\lambda}{2\nu}\right)}\left|\mathrm{IAi}\left(x\right)\right|^{2}du\gtrsim\left(-\Theta\left(\frac{\lambda}{2\nu}\right)\right)^{-1/2}\gtrsim\left(1-\frac{\lambda}{2\nu}\right)^{-1/2}.
\end{align*}
Collecting the above estimates, we obtain
\[
\int_{\sqrt{\frac{\nu}{\lambda}}}^{1}\left(1-\frac{\rho^{2}\lambda}{2\nu}\right)^{-3/2}\left|\Psi\left(\rho^{2}\lambda,\nu\right)\right|^{2}d\rho\gtrsim\left(1-\frac{\lambda}{2\nu}\right)^{-1/2}-\varepsilon\left(1-\frac{\lambda}{2\nu}\right)^{-1/2}\gtrsim\left(1-\frac{\lambda}{2\nu}\right)^{-1/2},
\]
for $\varepsilon$ small enough. Finally, let $2\left(\nu-\nu^{1/3}\right)<\lambda\leqslant2\left(\nu+2\nu^{1/3}\right)$.
We have
\begin{align*}
\int_{0}^{1}\left|\widehat{\chi_{B_{\rho}}}\left(\lambda,k\right)\right|^{2}d\rho & \geqslant\int_{0}^{\sqrt{\frac{2\left(\nu-\nu^{1/3}\right)}{\lambda}}}\rho^{2Q}\left|\widehat{\chi_{B}}\left(\rho^{2}\lambda,k\right)\right|^{2}d\rho\\
 & \gtrsim\int_{0}^{1}u^{2Q}\left|\widehat{\chi_{B}}\left(u^{2}2\left(\nu-\nu^{1/3}\right),k\right)\right|^{2}du\gtrsim\lambda^{-Q-2/3},
\end{align*}
which, under the above assumption on $\lambda$, coincides with the
RHS of (\ref{eq: Lemma 21}).
\end{proof}
\begin{lem}
For $\nu$ large enough, if $\lambda>2\left(\nu+\nu^{1/3}\right)$,
then we have
\[
\int_{0}^{1}\left|\widehat{\chi_{B_{\rho}}}\left(\lambda,k\right)\right|^{2}d\rho\gtrsim\lambda^{-Q-1}\left(\lambda-2\nu\right).
\]
\end{lem}

\begin{proof}
Since for $\lambda>3\nu$ we have $O\left(\nu^{-\frac{n-1}{2}}\lambda^{n/2-1}e^{-c\lambda}\right)=O\left(\nu^{-\frac{2}{3}}\right)$,
we can combine parts (\ref{enu:thm c}) and (\ref{enu:thm d}) of
Theorem \ref{thm:stime}. As a result, for the whole range $\lambda>2\left(\nu+\nu^{1/3}\right)$,
\[
\widehat{\chi_{B}}\left(\lambda,k\right)=\left(-1\right)^{k}\frac{\sin\lambda}{\lambda^{n+1}}\left(\Omega\left(\lambda,\nu\right)+O\left(\frac{1}{\nu^{2/3}}\right)\right),
\]
for a suitable $\Omega\left(\lambda,\nu\right)\approx1$. Hence
\begin{align*}
\int_{0}^{1}\left|\widehat{\chi_{B_{\rho}}}\left(\lambda,k\right)\right|^{2}d\rho\geqslant & \int_{\sqrt{\frac{2\left(\nu+\nu^{1/3}\right)}{\lambda}}}^{1}\rho^{2Q}\left|\widehat{\chi_{B}}\left(\rho^{2}\lambda,k\right)\right|^{2}d\rho\\
\gtrsim & \frac{1}{\lambda^{Q}}\int_{\sqrt{\frac{2\left(\nu+\nu^{1/3}\right)}{\lambda}}}^{1}\left|\sin\left(\rho^{2}\lambda\right)\right|^{2}d\rho+O\left(\frac{1}{\lambda^{Q}}\frac{1}{\nu^{4/3}}\right)\\
\gtrsim & \frac{1}{\lambda^{Q+1/2}}\int_{2\left(\nu+\nu^{1/3}\right)}^{\lambda}\left|\sin\left(u\right)\right|^{2}\frac{1}{\sqrt{u}}du+O\left(\frac{1}{\lambda^{Q}}\frac{1}{\nu^{4/3}}\right)\\
\gtrsim & \frac{1}{\lambda^{Q}}\left(1-\frac{2\nu}{\lambda}\right)+O\left(\frac{1}{\lambda^{Q}}\frac{1}{\nu^{4/3}}\right)\gtrsim\lambda^{-Q-1}\left(\lambda-2\nu\right),
\end{align*}
for $\nu$ large enough.
\end{proof}

\subsection{Estimate for $\mathcal{I}\left(\Lambda,s\right)$.}

In this subsection, we restrict the lower bounds obtained in the previous
section to the region $F_{\Lambda}$. This allows us to derive the
lower estimate for the term $\mathcal{I}\left(\Lambda,s\right)$ appearing
in equation (\ref{eq:I x J}).
\begin{thm}
\label{thm: stima I}Let $0<s<1$ and let $\Lambda>0$ be such that
$s\Lambda>Q-1$. We have
\[
\mathcal{I}\left(\Lambda,s\right)=\inf_{\left(k,\lambda\right)\in F_{\Lambda}}e^{\frac{1}{2}\nu\left|\lambda\right|s}\int_{0}^{1}\left|\widehat{\chi_{B_{\rho}}}\left(\lambda,k\right)\right|^{2}d\rho\gtrsim s^{\frac{Q-1}{2}}.
\]
\end{thm}

\begin{proof}
Recall that
\[
F_{\Lambda}=\left\{ \left(\lambda,k\right)\in\mathbb{R}^{*}\times\mathbb{N}:\left\langle \lambda\nu\right\rangle \leqslant\Lambda,\left|\lambda\right|\leqslant1\right\} .
\]
By Proposition \ref{prop: stima media-1}, we have
\begin{align*}
\inf_{\substack{\left(\lambda,k\right)\in F_{\Lambda}\\
0<\left|\lambda\right|\leqslant1
}
}\left(e^{\frac{1}{2}\nu\left|\lambda\right|s}\int_{0}^{1}\left|\widehat{\chi_{B_{\rho}}}\left(\lambda,k\right)\right|^{2}d\rho\right)\gtrsim & \inf_{\substack{\left\langle \lambda\nu\right\rangle \leqslant\Lambda\\
0<\left|\lambda\right|\leqslant1
}
}e^{\frac{1}{2}\nu\left|\lambda\right|s}\left\langle \lambda\nu\right\rangle ^{-\frac{Q-1}{2}}\\
\gtrsim & \inf_{\substack{\left\langle x\right\rangle \leqslant\Lambda\\
0<x\leqslant\nu
}
}e^{\frac{1}{2}\left\langle x\right\rangle s}\left\langle x\right\rangle ^{-\frac{Q-1}{2}}=\inf_{1\leqslant y\leqslant\Lambda}e^{\frac{y}{2}s}y^{-\frac{Q-1}{2}}\approx s^{\frac{Q-1}{2}},
\end{align*}
where we used the fact that the function $g\left(y\right)=e^{\frac{y}{2}s}y^{-\frac{Q-1}{2}}$
attains its minimum in
\[
y_{0}=\frac{Q-1}{s}\leqslant\Lambda.
\]
\end{proof}

\section{\label{sec:Estimate J}Estimate for the term $\mathcal{J}$ in (\ref{eq:I x J})}

In this section, we obtain a lower bound for the term $\mathcal{J}$
in equation (\ref{eq:I x J}). To this end, we introduce a kernel
that is closely related to the heat kernel on the Heisenberg group.
More precisely, this kernel is defined as the convolution of the heat
kernel with a cutoff in the vertical variable, see (\ref{eq: conv Ks}).
Despite this modification, it retains several key properties of the
heat kernel, such as positivity, radial symmetry. We also derive upper
and lower bounds for this kernel.

In the second subsection, these estimates are used to control the
term $\mathcal{J}$. The integral defining $\mathcal{J}$ is decomposed
into an integral over the whole Heisenberg group, which provides the
main contribution, and a correction term given by the integral over
the complement of the region $F_{\Lambda}$.

\subsection{The kernel $K_{s}$}

Let $\Psi\in\mathcal{S}\left(\mathbb{R}\right)$ be a positive even
function with support in $\left[-\frac{1}{2},\frac{1}{2}\right]$
such that $\|\Psi\|_{L^{2}\left(\mathbb{R}\right)}=\sqrt{2\pi}$,
and let 
\[
F\left(t\right)=\left(\frac{1}{2\pi}\int_{-\infty}^{+\infty}\Psi\left(\lambda\right)e^{-i\lambda t}d\lambda\right)^{2}.
\]
Then
\[
\widehat{F}\left(\lambda\right)=\Psi*\Psi\left(\lambda\right)
\]
is positive and has support in the interval $\left[-1,1\right]$.
Moreover,
\[
\widehat{F}\left(\lambda\right)\leqslant\int_{-\infty}^{+\infty}F\left(t\right)dt=\frac{1}{\left(2\pi\right)^{2}}\int_{-\infty}^{+\infty}\left|\widehat{\Psi}\left(t\right)\right|^{2}dt=1.
\]
Let $s>0$ and let $q_{s}\left(z,t\right)$ denote the heat kernel
on $\mathbb{H}^{n}$ (see e.g. \cite[Section 2.8]{Thangavelu2004})
defined via its Fourier transform in $t$ as
\[
q_{s}^{\lambda}\left(z\right)=\sum_{k=0}^{+\infty}e^{-\left(2k+n\right)\left|\lambda\right|s}\varphi_{k}^{\lambda}\left(z\right),
\]
where $\varphi_{k}^{\lambda}$ are introduced in (\ref{eq:def phi lambda k}).
We define a new kernel on $\mathbb{H}^{n}$ by the one dimensional
convolution
\begin{equation}
K_{s}\left(z,t\right)=\int_{-\infty}^{+\infty}F\left(t-\tau\right)q_{s}\left(z,\tau\right)d\tau.\label{eq: conv Ks}
\end{equation}
In the following lemma we collect some of the main properties of $K_{s}$
that will be used throughout the paper.
\begin{lem}
\label{lem:prop K_s}The kernel $K_{s}$ is a symmetric radial function
on $\mathbb{H}^{n}$ and its Fourier coefficients satisfies
\begin{align}
\widehat{K}_{s}\left(\lambda,k\right) & \leqslant e^{-\left(2k+n\right)\left|\lambda\right|s}, &  & \text{if }\left|\lambda\right|\leqslant1,\label{eq:coeff K_Lambda,s}\\
\widehat{K}_{s}\left(\lambda,k\right) & =0, &  & \text{if }\left|\lambda\right|>1.\label{eq:supp Ks}
\end{align}
If $s<1$, then for every $M>0$ there exist positive constants $C$
and $A$ such that for every $\left(z,t\right)\in\mathbb{H}^{n}$
we have
\begin{equation}
0\leqslant K_{s}\left(z,t\right)\leqslant Cs^{-n}e^{-\frac{A}{s}\left|z\right|^{2}}\left(1+\left|t\right|\right)^{-M}.\label{eq:stima K_Lambda,s}
\end{equation}
Furthermore
\[
K_{s}\left(0,0\right)\geqslant cs^{-n}.
\]
\end{lem}

\begin{proof}
The fact that $K_{s}\left(-z,-t\right)=K_{s}\left(z,t\right)$ follows
from the properties of the heat kernel and the fact that the function
$\Psi$ is even. Observe now that, by Theorem \ref{thm:Fourier transf radial},
\[
\widehat{K_s}(\lambda,k)=\widehat{F}(\lambda) e^{-(2k+n)|\lambda|s}.
\]
Since 
$0\leqslant\left|\widehat{F}\left(\lambda\right)\right|\leqslant1$
and ${\rm supp}\widehat{F}\subseteq\left[-1,1\right]$, (\ref{eq:coeff K_Lambda,s}) and (\ref{eq:supp Ks}) follows. Furthermore,
since both $q_{s}\left(z,t\right)$ and $F\left(t\right)$ are positive
the same holds for $K_{s}$. To prove (\ref{eq:stima K_Lambda,s}),
observe that, since $F\in\mathcal{S}\left(\mathbb{R}\right)$, we
have
\[
F\left(t\right)\leqslant\frac{1}{\left(1+\left|t\right|\right)^{M}}
\]
and that by Proposition 2.8.2 in \cite{Thangavelu2004}, there exist
$A,C>0$ such that
\[
q_{s}\left(z,\tau\right)\leqslant Cs^{-n-1}e^{-\frac{A}{s}\left(\left|z\right|^{2}+\left|\tau\right|\right)}.
\]
Hence
\begin{align*}
K_{s}\left(z,t\right)\leqslant & Cs^{-n-1}e^{-\frac{A}{s}\left|z\right|^{2}}\int_{-\infty}^{+\infty}\frac{1}{\left(1+\left|t-\tau\right|\right)^{M}}e^{-\frac{A}{s}\left|\tau\right|}d\tau.
\end{align*}
Let us consider
\begin{align}
\int_{-\infty}^{+\infty}\frac{1}{\left(1+\left|t-\tau\right|\right)^{M}}e^{-\frac{A}{s}\left|\tau\right|}d\tau= & \frac{s}{A}\int_{-\infty}^{+\infty}\frac{1}{\left(1+\frac{s}{A}\left|\frac{A}{s}t-\omega\right|\right)^{M}}e^{-\left|\omega\right|}d\omega=\frac{s}{A}H_{\frac{s}{A}}\left(\frac{A}{s}t\right),\label{eq: H_(s/A)}
\end{align}
where
\[
H_{\alpha}\left(t\right)=\int_{-\infty}^{+\infty}\frac{1}{\left(1+\alpha\left|t-\omega\right|\right)^{M}}e^{-\left|\omega\right|}d\omega.
\]
Clearly, for every $t\in\mathbb{R}$ we have 
\[
H_{\alpha}\left(t\right)\leqslant\int_{-\infty}^{+\infty}e^{-\left|\omega\right|}d\omega=2.
\]
Since $H_{\alpha}$ is even we can assume $t>0$. Then
\begin{align*}
H_{\alpha}\left(t\right)= & \int_{-\infty}^{\frac{\alpha}{\alpha+1}t}\frac{1}{\left(1+\alpha\left|t-w\right|\right)^{M}}e^{-\left|\omega\right|}d\omega+\int_{\frac{\alpha}{\alpha+1}t}^{+\infty}\frac{1}{\left(1+\alpha\left|t-w\right|\right)^{M}}e^{-\left|\omega\right|}d\omega\\
\leqslant & \left\Vert \left(1+\alpha\left|t-\cdot\right|\right)^{-M}\right\Vert _{L^{\infty}\left(-\infty,\frac{\alpha}{\alpha+1}t\right)}\int_{-\infty}^{\frac{\alpha}{\alpha+1}t}e^{-\left|\omega\right|}d\omega\\
 & +\left\Vert e^{-\frac{1}{2}\left|\cdot\right|}\right\Vert _{L^{\infty}\left(\frac{\alpha}{\alpha+1}t,+\infty\right)}\int_{\frac{\alpha}{\alpha+1}t}^{+\infty}\frac{e^{-\frac{1}{2}\left|\omega\right|}}{\left(1+\alpha\left|t-w\right|\right)^{M}}d\omega\\
\leqslant & c\frac{1}{\left(1+\frac{\alpha}{\alpha+1}t\right)^{M}}+e^{-\frac{1}{2}\frac{\alpha}{\alpha+1}t}\leqslant\frac{c}{\left(1+\frac{\alpha}{\alpha+1}t\right)^{M}}.
\end{align*}
Therefore, since $s<1$, by (\ref{eq: H_(s/A)}), we have
\begin{align*}
K_{s}\left(z,t\right)\lesssim & s^{-n}e^{-\frac{A}{s}\left|z\right|^{2}}\left(1+\frac{\frac{s}{A}}{\frac{s}{A}+1}\left|\frac{A}{s}t\right|\right)^{-M}\lesssim s^{-n}e^{-\frac{A}{s}\left|z\right|^{2}}\left(1+\left|t\right|\right)^{-M}.
\end{align*}
Finally, by Theorem 1 in \cite{Li2007}, we have $q_{1}\left(0,\tau\right)\geqslant Be^{-A\left|\tau\right|}$.
Using (2.8.6) in \cite{Thangavelu2004}, this implies that
\[
q_{s}\left(0,\tau\right)\geqslant s^{-n-1}Be^{-As^{-1}\left|\tau\right|}.
\]
We fix $\varepsilon>0$ such that $F\left(t\right)>\frac{1}{2}$ if
$\left|t\right|<\varepsilon$. Since $s<1,$ we have 
\begin{align*}
K_{s}\left(0,0\right)= & \int_{-\infty}^{+\infty}F\left(-\tau\right)q_{s}\left(0,\tau\right)d\tau\gtrsim s^{-n-1}\int_{0}^{\varepsilon}e^{-As^{-1}\tau}d\tau\gtrsim s^{-n}.
\end{align*}
\end{proof}

\subsection{Estimate from below of $\mathcal{J}$.}

To estimate $\mathcal{J}$, we split it into several parts. Using
(\ref{eq:coeff K_Lambda,s}) we write,
\begin{align*}
\mathcal{J}\geqslant & \sum_{k=0}^{+\infty}\int_{-\infty}^{+\infty}\chi_{F_{\Lambda}}\left(k,\lambda\right)\widehat{K_{s}}\left(\lambda,k\right)\sum_{\left|\alpha\right|=k}\left\Vert \widehat{\sigma}\left(\lambda\right)\Phi_{\alpha}^{\lambda}\right\Vert _{L^{2}\left(\mathbb{R}^{n}\right)}^{2}\left|\lambda\right|^{n}d\lambda\\
= & \sum_{k=0}^{+\infty}\int_{-\infty}^{+\infty}\widehat{K_{s}}\left(\lambda,k\right)\sum_{\left|\alpha\right|=k}\left\Vert \widehat{\sigma}\left(\lambda\right)\Phi_{\alpha}^{\lambda}\right\Vert _{L^{2}\left(\mathbb{R}^{n}\right)}^{2}\left|\lambda\right|^{n}d\lambda\\
 & -\sum_{k=0}^{+\infty}\int_{-\infty}^{+\infty}\chi_{F_{\Lambda}^{C}\left(k,\lambda\right)}\widehat{K_{s}}\left(\lambda,k\right)\sum_{\left|\alpha\right|=k}\left\Vert \widehat{\sigma}\left(\lambda\right)\Phi_{\alpha}^{\lambda}\right\Vert _{L^{2}\left(\mathbb{R}^{n}\right)}^{2}\left|\lambda\right|^{n}d\lambda\\
= & A-B.
\end{align*}

\subsubsection{Estimate of $A$}

To estimate the term $A$ we exploit global properties of the kernel
$K_{s}$ and the fact that the measure $\mu$ is upper Ahlfors regular. 

Since
\begin{align}
\left\Vert \widehat{\sigma}\left(\lambda\right)\Phi_{\alpha}^{\lambda}\right\Vert _{L^{2}\left(\mathbb{R}^{n}\right)}^{2}= & \int_{\mathbb{R}^{n}}\left|\sum_{j=1}^{N}\widehat{\delta}_{\left(\rho_{j},\tau_{j}\right)}\left(\lambda\right)\Phi_{\alpha}^{\lambda}-N\widehat{\mu}\left(\lambda\right)\Phi_{\alpha}^{\lambda}\right|^{2}d\xi\label{eq:sigma hat}\\
\geqslant & \int_{\mathbb{R}^{n}}\left|\sum_{j=1}^{N}\widehat{\delta}_{\left(\rho_{j},\tau_{j}\right)}\left(\lambda\right)\Phi_{\alpha}^{\lambda}\right|^{2}d\xi-2N\sum_{j=1}^{N}\mathrm{Re}\int_{\mathbb{R}^{n}}\widehat{\delta}_{\left(\rho_{j},\tau_{j}\right)}\left(\lambda\right)\Phi_{\alpha}^{\lambda}\overline{\widehat{\mu}\left(\lambda\right)\Phi_{\alpha}^{\lambda}}d\xi,\nonumber 
\end{align}
we can split $A$ as 
\begin{align*}
A\geqslant & \sum_{k=0}^{+\infty}\int_{-\infty}^{+\infty}\widehat{K_{s}}\left(\lambda,k\right)\sum_{\left|\alpha\right|=k}\int_{\mathbb{R}^{n}}\left|\sum_{j=1}^{N}\widehat{\delta}_{\left(\rho_{j},\tau_{j}\right)}\left(\lambda\right)\Phi_{\alpha}^{\lambda}\right|^{2}d\xi\left|\lambda\right|^{n}d\lambda\\
 & -2N\sum_{k=0}^{+\infty}\int_{-\infty}^{+\infty}\widehat{K_{s}}\left(\lambda,k\right)\left(\sum_{j=1}^{N}\mathrm{Re}\sum_{\left|\alpha\right|=k}\int_{\mathbb{R}^{n}}\widehat{\delta}_{\left(\rho_{j},\tau_{j}\right)}\left(\lambda\right)\Phi_{\alpha}^{\lambda}\overline{\widehat{\mu}\left(\lambda\right)\Phi_{\alpha}^{\lambda}}d\xi\right)\left|\lambda\right|^{n}d\lambda\\
= & A_{1}-A_{2}.
\end{align*}
We now estimate $A_{2}$ from above. We start be rewriting it as follows.
\begin{lem}
There exists a constant $\tilde{c}>0$ such that
\[
A_{2}=\tilde{c}N\sum_{j=1}^{N}\mu*K_{s}\left(\rho_{j},\tau_{j}\right).
\]
\end{lem}

\begin{proof}
Using (\ref{eq: pi lambde}) and Proposition \ref{prop: proprieta Shrodinger},
we write
\begin{align*}
\int_{\mathbb{R}^{n}}\widehat{\delta}_{\left(\rho_{j},\tau_{j}\right)}\left(\lambda\right)\Phi_{\alpha}^{\lambda}\overline{\widehat{\mu}\left(\lambda\right)\Phi_{\alpha}^{\lambda}}d\xi & =\left\langle \pi_{\lambda}\left(\rho_{j},\tau_{j}\right)\Phi_{\alpha}^{\lambda},\widehat{\mu}\left(\lambda\right)\Phi_{\alpha}^{\lambda}\right\rangle \\
 & =\int_{\mathbb{H}^{n}}\left\langle \pi_{\lambda}\left(\rho_{j},\tau_{j}\right)\Phi_{\alpha}^{\lambda},\pi_{\lambda}\left(z,t\right)\Phi_{\alpha}^{\lambda}\right\rangle d\mu\left(z,t\right)\\
 & =\int_{\mathbb{H}^{n}}\left\langle \pi_{\lambda}\left(\left(z,t\right)^{-1}\circ\left(\rho_{j},\tau_{j}\right)\right)\Phi_{\alpha}^{\lambda},\Phi_{\alpha}^{\lambda}\right\rangle d\mu\left(z,t\right)\\
 & =\int_{\mathbb{H}^{n}}e^{i\lambda\left(\tau_{j}-t-\frac{1}{2}\mathtt{Im}z\cdot\overline{\rho_{j}}\right)}\left\langle \pi_{\lambda}\left(\rho_{j}-z\right)\Phi_{\alpha}^{\lambda},\Phi_{\alpha}^{\lambda}\right\rangle d\mu\left(z,t\right).
\end{align*}
Hence, from Proposition \ref{prop:phi_k}, we get 
\begin{align*}
\sum_{\left|\alpha\right|=k}\int_{\mathbb{R}^{n}}\widehat{\delta}_{\left(\rho_{j},\tau_{j}\right)}\left(\lambda\right)\Phi_{\alpha}^{\lambda}\overline{\widehat{\nu}\left(\lambda\right)\Phi_{\alpha}^{\lambda}}d\xi= & \int_{\mathbb{H}^{n}}e^{i\lambda\left(\tau_{j}-t-\frac{1}{2}\mathtt{Im}z\cdot\overline{\rho_{j}}\right)}\left(\frac{2\pi}{\left|\lambda\right|}\right)^{n}\varphi_{k}^{\lambda}\left(\rho_{j}-z\right)d\mu\left(z,t\right).
\end{align*}
Finally, by Theorem \ref{thm: ricostruzione}, we have
\begin{align*}
A_{2}\simeq & N\sum_{k=0}^{+\infty}\int_{-\infty}^{+\infty}\widehat{K_{s}}\left(\lambda,k\right)\left(\sum_{j=1}^{N}\mathrm{Re}\int_{\mathbb{H}^{n}}e^{i\lambda\left(\tau_{j}-t-\frac{1}{2}\mathtt{Im}z\cdot\overline{\rho_{j}}\right)}\varphi_{k}^{\lambda}\left(\rho_{j}-z\right)d\mu\left(z,t\right)\right)d\lambda\\
= & N\sum_{j=1}^{N}\mathrm{Re}\int_{\mathbb{H}^{n}}\int_{-\infty}^{+\infty}e^{i\lambda\left(\tau_{j}-t-\frac{1}{2}\mathtt{Im}z\cdot\overline{\rho_{j}}\right)}\sum_{k=0}^{+\infty}\widehat{K_{s}}\left(\lambda,k\right)\varphi_{k}^{\lambda}\left(\rho_{j}-z\right)d\lambda d\mu\left(z,t\right)\\
= & \tilde{c}N\sum_{j=1}^{N}\int_{\mathbb{H}^{n}}K_{s}\left(\rho_{j}-z,-\left(\tau_{j}-t-\frac{1}{2}\mathtt{Im}z\cdot\overline{\rho_{j}}\right)\right)d\mu\left(z,t\right)\\
= & \tilde{c}N\sum_{j=1}^{N}\int_{\mathbb{H}^{n}}K_{s}\left(\left(\rho_{j},\tau_{j}\right)^{-1}\circ\left(z,t\right)\right)d\mu\left(z,t\right)\\
= & \tilde{c}N\sum_{j=1}^{N}\int_{\mathbb{H}^{n}}K_{s}\left(\left(z,t\right)^{-1}\circ\left(\rho_{j},\tau_{j}\right)\right)d\mu\left(z,t\right)\\
= & \tilde{c}N\sum_{j=1}^{N}\mu*K_{s}\left(\rho_{j},\tau_{j}\right).
\end{align*}
\end{proof}
To estimate the term $A_{2}$, we proceed in two steps. We first introduce
a technical lemma, which is then followed by a result providing an
upper bound for the convolution appearing above.
\begin{lem}
\label{lem: mu<ab}Let $\alpha,\beta\geqslant0$ and let $Q_{\alpha,\beta}=\left\{ \left(z,t\right)\in\mathbb{H}^{n}:\left|z\right|\leqslant\alpha,\left|t\right|\leqslant\beta\right\} $.
There exists $c\geqslant0$ such that for every $\left(\rho,\tau\right)\in\mathbb{H}^{n}$
and $\alpha,\beta\geqslant0$
\[
\mu\left(\left(\rho,\tau\right)\circ Q_{\alpha,\beta}\right)\leqslant c\alpha^{2n}\beta.
\]
\end{lem}

\begin{proof}
Let $r$ be a small positive number and let $\left\{ \mathfrak{B}_{j}\right\} _{j=1}^{M}$
be a maximal collection of disjoint Heisenberg balls of radius $r$
with center in $Q_{\alpha,\beta}$. Observe that there exists a constant
$c_{n}$ such that if $r$ is small enough then $\left|\mathfrak{B}_{j}\cap Q_{\alpha,\beta}\right|\geqslant c_{n}r^{Q}$
for every $j=1,\dots,M$. Hence,
\begin{equation}
c_{n}Mr^{Q}\leqslant\left|\bigcup_{j=1}^{M}\left(\mathfrak{B}_{j}\cap Q_{\alpha,\beta}\right)\right|\leqslant\left|Q_{\alpha,\beta}\right|=\alpha^{2n}\beta.\label{eq: M<ab}
\end{equation}
Let $\mathfrak{B}_{j}'$ be the ball of radius $3r$ with the same
center of $\mathfrak{B}_{j}$. The family $\left\{ \mathfrak{B}_{j}'\right\} _{j=1}^{M}$
covers $Q_{\alpha,\beta}$. Since $\mu$ is upper Ahlfors regular,
\ref{eq: M<ab} implies that
\[
\mu\left(\left(\rho,\tau\right)\circ Q_{\alpha,\beta}\right)\leqslant\mu\left(\bigcup_{j=1}^{M}\left(\rho,\tau\right)\circ\mathfrak{B}_{j}'\right)\leqslant\sum_{j=1}^{M}\mu\left(\left(\rho,\tau\right)\circ\mathfrak{B}_{j}'\right)\lesssim Mr^{Q}\lesssim\alpha^{2n}\beta.
\]
\end{proof}
\begin{lem}
There exists $C>0$ such that for every $\left(\rho,\tau\right)\in\mathbb{H}^{n}$,
we have
\[
\mu*K_{s}\left(\rho,\tau\right)\leqslant C.
\]
\end{lem}

\begin{proof}
For any $k,j\in\mathbb{N}$, let
\[
E_{k,j}=\left\{ \left(z,t\right)\in\mathbb{H}^{n}:2^{k-1}<\frac{1}{s}\left|z\right|^{2}\leqslant2^{k}\text{ and }2^{j-1}<\left|t\right|\leqslant2^{j}\right\} .
\]
where, as a notational convention, we interpret $2^{-1}$ as $0$.
Using (\ref{eq:stima K_Lambda,s}) for $M>1$, we can write
\begin{align*}
K_{s}\left(z,t\right)\leqslant & Cs^{-n}e^{-\frac{A}{s}\left|z\right|^{2}}\left(1+\left|t\right|\right)^{-M}\\
\leqslant & Cs^{-n}\left(1+\frac{1}{s}\left|z\right|^{2}\right)^{-nM}\left(1+\left|t\right|\right)^{-M}\\
\leqslant & Cs^{-n}\sum_{k=0}^{+\infty}\sum_{j=0}^{+\infty}\chi_{E_{k,j}}\left(z,t\right)\left(1+2^{k-1}\right)^{-nM}\left(1+2^{j-1}\right)^{-M}.
\end{align*}
Since $E_{k,j}\subseteq Q_{\sqrt{s2^{k}},2^{j}}$, by Lemma \ref{lem: mu<ab},
we have
\[
\mu\left(\left(\rho,\tau\right)\circ E_{k,j}\right)\leqslant\mu\left(\left(\rho,\tau\right)\circ Q_{\sqrt{s2^{k}},2^{j}}\right)\leqslant c\left(\sqrt{s2^{k}}\right)^{2n}2^{j}=cs^{n}2^{kn}2^{j}.
\]
Therefore, using that $-E_{k,j}=E_{k,j}$, we obtain
\begin{align*}
\mu*K_{s}\left(\rho,\tau\right)= & \int_{\mathbb{H}^{n}}K_{s}\left(\left(z,t\right)^{-1}\circ\left(\rho,\tau\right)\right)d\mu\left(z,t\right)\\
\leqslant & Cs^{-n}\sum_{k=0}^{+\infty}\sum_{j=0}^{+\infty}\int_{\mathbb{H}^{n}}\chi_{E_{k,j}}\left(\left(z,t\right)^{-1}\circ\left(\rho,\tau\right)\right)\left(1+2^{k-1}\right)^{-nM}\left(1+2^{j-1}\right)^{-M}d\mu\left(z,t\right)\\
= & Cs^{-n}\sum_{k=0}^{+\infty}\sum_{j=0}^{+\infty}\left(1+2^{k-1}\right)^{-nM}\left(1+2^{j-1}\right)^{-M}\mu\left(\left(\rho,\tau\right)\circ E_{k,j}\right)\\
\lesssim & \sum_{k=0}^{+\infty}\sum_{j=0}^{+\infty}\left(1+2^{k-1}\right)^{-nM}\left(1+2^{j-1}\right)^{-M}2^{nk}2^{j}\lesssim1.
\end{align*}
\end{proof}
Using the above lemma we obtain
\begin{equation}
A_{2}\leqslant CN^{2}\label{eq: A2}
\end{equation}
 for a suitable constant $C>0$. To estimate $A_{1}$ we start with
the following lemma.
\begin{lem}
\label{lem:delta L2}Let $\left\{ \left(\rho_{j},\tau_{j}\right)\right\} _{j=1}^{N}$
be a finite sequence of points in $\mathbb{H}^{n},$ then
\begin{align*}
\sum_{\left|\alpha\right|=k}\left\Vert \sum_{j=1}^{N}\widehat{\delta}_{\left(\rho_{j},\tau_{j}\right)}\left(\lambda\right)\Phi_{\alpha}^{\lambda}\right\Vert _{L^{2}\left(\mathbb{R}^{n}\right)}^{2} & =\left(\frac{2\pi}{\left|\lambda\right|}\right)^{n}\sum_{\ell,j=1}^{N}e^{i\lambda\left(\tau_{j}-\tau_{\ell}-\frac{1}{2}\mathrm{Im}\rho_{\ell}\overline{\rho_{j}}\right)}\varphi_{\ell}^{\lambda}\left(\rho_{j}-\rho_{\ell}\right).
\end{align*}
\end{lem}

\begin{proof}
Since
\[
\widehat{\delta}_{\left(\rho_{j},\tau_{j}\right)}\left(\lambda\right)\Phi_{\alpha}^{\lambda}=\pi_{\lambda}\left(\rho_{j},\tau_{j}\right)\Phi_{\alpha}^{\lambda},
\]
using Proposition \ref{prop: proprieta Shrodinger} and Proposition
\ref{prop:phi_k} we have 
\begin{align*}
\sum_{\left|\alpha\right|=k}\left\Vert \sum_{j=1}^{N}\widehat{\delta}_{\left(\rho_{j},\tau_{j}\right)}\left(\lambda\right)\Phi_{\alpha}^{\lambda}\right\Vert _{L^{2}\left(\mathbb{R}^{n}\right)}^{2}= & \sum_{\left|\alpha\right|=k}\sum_{\ell,j=1}^{N}\left\langle \pi_{\lambda}\left(\rho_{j},\tau_{j}\right)\Phi_{\alpha}^{\lambda},\pi_{\lambda}\left(\rho_{\ell},\tau_{\ell}\right)\Phi_{\alpha}^{\lambda}\right\rangle \\
= & \sum_{\ell,j=1}^{N}\sum_{\left|\alpha\right|=k}\left\langle \pi_{\lambda}\left(\rho_{j}-\rho_{\ell},\tau_{j}-\tau_{\ell}-\frac{1}{2}\mathrm{Im}\rho_{\ell}\overline{\rho_{j}}\right)\Phi_{\alpha}^{\lambda},\Phi_{\alpha}^{\lambda}\right\rangle \\
= & \sum_{\ell,j=1}^{N}e^{i\lambda\left(\tau_{j}-\tau_{\ell}-\frac{1}{2}\mathrm{Im}\rho_{\ell}\overline{\rho_{j}}\right)}\left(\frac{2\pi}{\left|\lambda\right|}\right)^{n}\varphi_{\ell}^{\lambda}\left(\rho_{j}-\rho_{\ell}\right).
\end{align*}
\end{proof}
\begin{prop}
We have
\[
A_{1}\geqslant CNs^{-n}
\]
where $C>0$ only depends on the dimension $n$.
\end{prop}

\begin{proof}
Lemma \ref{lem:delta L2} and the fact that $K_{s}\left(z,t\right)\geqslant0$
yield
\begin{align*}
A_{1}= & \left(2\pi\right)^{n}\sum_{\ell,j=1}^{N}\int_{-\infty}^{+\infty}e^{i\lambda\left(\tau_{j}-\tau_{\ell}-\frac{1}{2}\mathrm{Im}\rho_{\ell}\overline{\rho_{j}}\right)}\sum_{k=0}^{+\infty}\widehat{K_{s}}\left(\lambda,k\right)\varphi_{\ell}^{\lambda}\left(\rho_{j}-\rho_{\ell}\right)d\lambda\\
= & c\sum_{\ell,j=1}^{N}K_{s}\left(\rho_{j}-\rho_{\ell},-\left(\tau_{j}-\tau_{\ell}-\frac{1}{2}\mathrm{Im}\rho_{\ell}\overline{\rho_{j}}\right)\right)\\
\geqslant & c\sum_{\ell=1}^{N}K_{s}\left(0,0\right)\geqslant cNs^{-n}.
\end{align*}
\end{proof}
By the above lemma and (\ref{eq: A2}) we obtain
\begin{equation}
A\geqslant C_{1}Ns^{-n}-C_{2}N^{2}.\label{eq:Stima A}
\end{equation}

\subsubsection{Estimate of $B$}

To complete the proof of the estimate in (\ref{eq:stima J}), it remains
to control from above the term $B$. This is achieved by exploiting
the exponential decay of the kernel $K_{s}$ outside the region $F_{\Lambda}$.
In this sense, the integration in $B$ represents a tail contribution
which, for a suitable choice of parameters, turns out to be negligible
in the overall estimate of $\mathcal{J}$.
\begin{prop}
\label{prop:Stima B}Assume $s\Lambda>1$, then
there exists $C>0$ such that
\[
\left|B\right|\leqslant CN^{2}s^{-1}\Lambda^{\frac{Q-4}2}e^{-\frac{1}{2}s\Lambda}.
\]
\end{prop}

\begin{proof}
Let us consider
\begin{align*}
\left\Vert \widehat{\sigma}\left(\lambda\right)\Phi_{\alpha}^{\lambda}\right\Vert _{L^{2}\left(\mathbb{R}^{n}\right)}^{2}= & \left\Vert \sum_{j=1}^{N}\widehat{\delta}_{\left(\rho_{j},\tau_{j}\right)}\left(\lambda\right)\Phi_{\alpha}^{\lambda}-N\widehat{\mu}\left(\lambda\right)\Phi_{\alpha}^{\lambda}\right\Vert _{L^{2}\left(\mathbb{R}^{n}\right)}^{2}\\
\leqslant & 2\left\Vert \sum_{j=1}^{N}\widehat{\delta}_{\left(\rho_{j},\tau_{j}\right)}\left(\lambda\right)\Phi_{\alpha}^{\lambda}\right\Vert _{L^{2}\left(\mathbb{R}^{n}\right)}^{2}+2\left\Vert N\widehat{\mu}\left(\lambda\right)\Phi_{\alpha}^{\lambda}\right\Vert _{L^{2}\left(\mathbb{R}^{n}\right)}^{2}.
\end{align*}
By Lemma \ref{lem:delta L2} and (\ref{eq:def phi lambda k}) and
\cite[18.14.8]{NIST}, we have
\begin{align*}
\sum_{\left|\alpha\right|=k}\left\Vert \sum_{j=1}^{N}\widehat{\delta}_{\left(\rho_{j},\tau_{j}\right)}\left(\lambda\right)\Phi_{\alpha}^{\lambda}\right\Vert _{L^{2}\left(\mathbb{R}^{n}\right)}^{2}= & \left(\frac{2\pi}{\left|\lambda\right|}\right)^{n}\sum_{\ell,j=1}^{N}e^{i\lambda\left(\tau_{j}-\tau_{\ell}-\frac{1}{2}\mathrm{Im}\rho_{\ell}\overline{\rho_{j}}\right)}\varphi_{\ell}^{\lambda}\left(\rho_{j}-\rho_{\ell}\right)\\
\leqslant & \sum_{\ell,j=1}^{N}\left|L_{k}^{n-1}\left(\frac{1}{2}\left|\lambda\right|\left|\rho_{j}-\rho_{\ell}\right|^{2}\right)\right|e^{-\frac{1}{4}\left|\lambda\right|\left|\rho_{j}-\rho_{\ell}\right|^{2}}\\
\leqslant & N^{2}\sup_{x\in\left[0,\infty\right)}\left(\left|L_{k}^{n-1}\left(x\right)\right|e^{-\frac{1}{2}x}\right)=N^{2}L_{k}^{n-1}\left(0\right)=N^{2}\binom{n+k-1}{k}.
\end{align*}
Similarly,
\begin{align*}
\sum_{\left|\alpha\right|=k}\left\Vert \widehat{\mu}\left(\lambda\right)\Phi_{\alpha}^{\lambda}\right\Vert _{L^{2}\left(\mathbb{R}^{n}\right)}^{2}= & \int_{\mathbb{H}^{n}}\int_{\mathbb{H}^{n}}\sum_{\left|\alpha\right|=k}\left\langle \pi_{\lambda}\left(\left(w,u\right)^{-1}\circ\left(z,t\right)\right)\Phi_{\alpha}^{\lambda},\Phi_{\alpha}^{\lambda}\right\rangle d\mu\left(z,t\right)d\mu\left(w,u\right)\\
= & \int_{\mathbb{H}^{n}}\int_{\mathbb{H}^{n}}e^{i\lambda\left(t-u-\frac{1}{2}\mathtt{Im}w\cdot\overline{z}\right)}L_{k}^{n-1}\left(\frac{1}{2}\left|\lambda\right|\left|z-w\right|^{2}\right)e^{-\frac{1}{4}\left|\lambda\right|\left|z-w\right|^{2}}d\mu\left(z,t\right)d\mu\left(w,u\right)\\
\leqslant & \sup_{x\in\left[0,\infty\right)}\left(\left|L_{k}^{n-1}\left(x\right)\right|e^{-\frac{1}{2}x}\right)=\binom{n+k-1}{k}.
\end{align*}
Hence,
\begin{align*}
\sum_{\left|\alpha\right|=k}\left\Vert \widehat{\sigma}\left(\lambda\right)\Phi_{\alpha}^{\lambda}\right\Vert _{L^{2}\left(\mathbb{R}^{n}\right)}^{2}\leqslant & 4N^{2}\binom{n+k-1}{k}.
\end{align*}
By (\ref{eq:coeff K_Lambda,s}) and (\ref{eq:supp Ks}), we have 
\begin{align*}
B= & \sum_{k=0}^{+\infty}\int_{-\infty}^{+\infty}\chi_{F_{\Lambda}^{C}\left(k,\lambda\right)}\widehat{K_{s}}\left(\lambda,k\right)\sum_{\left|\alpha\right|=k}\left\Vert \widehat{\sigma}\left(\lambda\right)\Phi_{\alpha}^{\lambda}\right\Vert _{L^{2}\left(\mathbb{R}^{n}\right)}^{2}\left|\lambda\right|^{n}d\lambda\\
\leqslant & 4N^{2}\sum_{k=0}^{+\infty}\binom{n+k-1}{k}\int_{-1}^{1}\chi_{F_{\Lambda}^{C}\left(k,\lambda\right)}e^{-\left(2k+n\right)\left|\lambda\right|s}\left|\lambda\right|^{n}d\lambda\\
\leqslant & 8N^{2}\sum_{\Lambda<\nu}\binom{n+k-1}{k}\int_{\nu^{-1}\left(\Lambda-1\right)}^{1}e^{-\frac{1}{2}\nu\lambda s}\lambda^{n}d\lambda\\
\lesssim & N^{2}\sum_{\Lambda<\nu}\nu^{n-1}\left(s\nu\right)^{-n-1}\int_{\frac{1}{2}s\left(\Lambda-1\right)}^{\frac{1}{2}\nu s}e^{-u}u^{n}d\lambda\\
\leqslant & N^{2}s^{-n-1}\Gamma\left(n+1,\frac{1}{2}s\left(\Lambda-1\right)\right)\sum_{\Lambda\leqslant\nu}\nu^{-2}\\
\lesssim & N^{2}s^{-n-1}\Lambda^{-1}\Gamma\left(n+1,\frac{1}{2}s\left(\Lambda-1\right)\right).
\end{align*}
Where $\Gamma\left(n+1,t\right)$ denotes the incomplete Gamma function.
Since $s\Lambda$ is large, we obtain
\[
B\lesssim N^{2}s^{-1}\Lambda^{\frac{Q-4}2}e^{-\frac{1}{2}s\Lambda}.
\]
\end{proof}
We conclude by observing that (\ref{eq:Stima A}), together with Proposition
\ref{prop:Stima B}, yields the estimates claimed in (\ref{eq:stima J}).

\bibliographystyle{abbrv}
\bibliography{biblio_Heisenberg}

@article{Li2007,
	author = {Li, Hong-Quan},
	doi = {10.1016/j.crma.2007.02.015},
	fjournal = {Comptes Rendus Math\'{e}matique. Acad\'{e}mie des Sciences. Paris},
	issn = {1631-073X},
	journal = {C. R. Math. Acad. Sci. Paris},
	mrclass = {35H10 (35B40 35B45 35K05 43A80 58J35 58J37)},
	mrnumber = {2324485},
	number = {8},
	pages = {497--502},
	title = {Estimations asymptotiques du noyau de la chaleur sur les groupes de {H}eisenberg},
	url = {https://doi.org/10.1016/j.crma.2007.02.015},
	volume = {344},
	year = {2007}}

@article{Skriganov2019,
	author = {Skriganov, M. M.},
	doi = {10.1112/s0025579319000019},
	fjournal = {Mathematika. A Journal of Pure and Applied Mathematics},
	issn = {0025-5793},
	journal = {Mathematika},
	mrclass = {11K38 (22F30 52C17 65D32)},
	mrnumber = {3932579},
	mrreviewer = {Christoph Aistleitner},
	number = {3},
	pages = {557--587},
	title = {Point distributions in two-point homogeneous spaces},
	url = {https://doi.org/10.1112/s0025579319000019},
	volume = {65},
	year = {2019}}

@article{Skriganov2017,
	author = {Skriganov, M. M.},
	doi = {10.1112/S0025579317000286},
	fjournal = {Mathematika. A Journal of Pure and Applied Mathematics},
	issn = {0025-5793},
	journal = {Mathematika},
	mrclass = {11K38 (49Q20)},
	mrnumber = {3731319},
	mrreviewer = {Stefan Steinerberger},
	number = {3},
	pages = {1152--1171},
	title = {Point distributions in compact metric spaces},
	url = {https://doi.org/10.1112/S0025579317000286},
	volume = {63},
	year = {2017}}

@book{DrmotaTichy1997,
	author = {Drmota, M. and Tichy, R. F.},
	doi = {10.1007/BFb0093404},
	isbn = {3-540-62606-9},
	mrclass = {11Kxx (11K06 11K38)},
	mrnumber = {1470456},
	mrreviewer = {Oto Strauch},
	pages = {xiv+503},
	publisher = {Springer-Verlag, Berlin},
	series = {Lecture Notes in Mathematics},
	title = {Sequences, discrepancies and applications},
	url = {https://doi.org/10.1007/BFb0093404},
	volume = {1651},
	year = {1997}}

@url{ChenLectures,
	author = {Chen, W. W. L.},
	title = {Lectures on irregularities of point distribution, web edition},
	urldate = {http://www.williamchen-mathematics.info/researchfolder/iod00.pdf}}

@book{Panorama2014,
	editor = {Chen, W. W. L. and Srivastav, A. and Travaglini, G.},
	isbn = {978-3-319-04695-2; 978-3-319-04696-9},
	mrclass = {11-06 (11J71 11K36 11K38)},
	mrnumber = {3307667},
	pages = {xvi+695},
	publisher = {Springer, Cham},
	series = {Lecture Notes in Mathematics},
	title = {A panorama of discrepancy theory},
	url = {https://mathscinet.ams.org/mathscinet-getitem?mr=3307667},
	volume = {2107},
	year = {2014}}

@book{ChazelleBook2000,
	author = {Chazelle, B.},
	doi = {10.1017/CBO9780511626371},
	isbn = {0-521-77093-9},
	mrclass = {68-02 (11K38 52B55 65Y20 68Q25 68U05)},
	mrnumber = {1779341},
	mrreviewer = {Allen D. Rogers},
	pages = {xviii+463},
	publisher = {Cambridge University Press, Cambridge},
	title = {The discrepancy method},
	url = {https://doi.org/10.1017/CBO9780511626371},
	year = {2000}}

@incollection{BrandoliniGiganteTravaglini2014,
	author = {Brandolini, L. and Gigante, G. and Travaglini, G.},
	booktitle = {A panorama of discrepancy theory},
	doi = {10.1007/978-3-319-04696-9\_3},
	mrclass = {11K38 (11K36 42B10)},
	mrnumber = {3330325},
	mrreviewer = {Robert F. Tichy},
	pages = {159--220},
	publisher = {Springer, Cham},
	series = {Lecture Notes in Math.},
	title = {Irregularities of distribution and average decay of {F}ourier transforms},
	url = {https://doi.org/10.1007/978-3-319-04696-9_3},
	volume = {2107},
	year = {2014}}

@book{Thangavelu2004,
	author = {Thangavelu, S.},
	doi = {10.1007/978-0-8176-8164-7},
	isbn = {0-8176-4330-3},
	mrclass = {43A80 (33C45 33C55 42C10 43-02 43A20)},
	mrnumber = {2008480},
	mrreviewer = {Weixing Zheng},
	pages = {xiv+174},
	publisher = {Birkh\"{a}user Boston, Inc., Boston, MA},
	series = {Progress in Mathematics},
	title = {An introduction to the uncertainty principle},
	url = {https://doi.org/10.1007/978-0-8176-8164-7},
	volume = {217},
	year = {2004}}

@book{NIST,
	author = {Olver, Frank W. and Lozier, Daniel W. and Boisvert, Ronald and Clark, Charles W.},
	publisher = {Cambridge University Press, New York, NY},
	title = {The NIST Handbook of Mathematical Functions},
	year = {2010}}

@article{AskeyWainger1965,
	author = {Askey, R. and Wainger, S.},
	doi = {10.2307/2373069},
	fjournal = {American Journal of Mathematics},
	issn = {0002-9327},
	journal = {Amer. J. Math.},
	mrclass = {42.20},
	mrnumber = {182834},
	mrreviewer = {G. M. Wing},
	pages = {695--708},
	title = {Mean convergence of expansions in {L}aguerre and {H}ermite series},
	url = {https://doi.org/10.2307/2373069},
	volume = {87},
	year = {1965}}

@article{FrenzenWong1988,
	author = {Frenzen, C. L. and Wong, R.},
	doi = {10.1137/0519087},
	fjournal = {SIAM Journal on Mathematical Analysis},
	issn = {0036-1410},
	journal = {SIAM J. Math. Anal.},
	mrclass = {33A65 (33A40 41A60)},
	mrnumber = {957682},
	mrreviewer = {F. W. J. Olver},
	number = {5},
	pages = {1232--1248},
	title = {Uniform asymptotic expansions of {L}aguerre polynomials},
	url = {https://doi.org/10.1137/0519087},
	volume = {19},
	year = {1988}}

@article{Erdelyi1960,
	author = {Erd\'{e}lyi, A.},
	fjournal = {The Journal of the Indian Mathematical Society. New Series},
	issn = {0019-5839},
	journal = {J. Indian Math. Soc. (N.S.)},
	mrclass = {33.40},
	mrnumber = {123751},
	mrreviewer = {T. M. Cherry},
	pages = {235--250 (1961)},
	title = {Asymptotic forms for {L}aguerre polynomials},
	url = {https://mathscinet.ams.org/mathscinet-getitem?mr=123751},
	volume = {24},
	year = {1960}}

@book{Matousek1999,
	author = {Matousek, J.},
	publisher = {Springer Science \& Business Media},
	title = {Geometric discrepancy: An illustrated guide},
	volume = {18},
	year = {1999}}

@article{BGG2021,
	author = {Brandolini, L. and Gariboldi, B. and Gigante, G.},
	fjournal = {Mathematische Annalen},
	issn = {0025-5831,1432-1807},
	journal = {Math. Ann.},
	number = {3-4},
	pages = {1807--1834},
	title = {On a sharp lemma of {C}assels and {M}ontgomery on manifolds},
	volume = {379},
	year = {2021}}

@book{TravagliniBook2014,
	author = {Travaglini, G.},
	isbn = {978-1-107-61985-2},
	pages = {x+240},
	publisher = {Cambridge University Press, Cambridge},
	series = {London Mathematical Society Student Texts},
	title = {Number theory, {F}ourier analysis and geometric discrepancy},
	volume = {81},
	year = {2014}}

@article{Roth1954,
	author = {Roth, K. F.},
	journal = {Mathematika},
	number = {2},
	pages = {73--79},
	publisher = {London Mathematical Society},
	title = {On irregularities of distribution},
	volume = {1},
	year = {1954}}

@book{BeckChenBook1987,
	author = {Beck, J. and Chen, W. W. L.},
	pages = {xiv+294},
	publisher = {Cambridge University Press, Cambridge},
	series = {Cambridge Tracts in Mathematics},
	title = {Irregularities of distribution},
	volume = {89},
	year = {1987}}

@inbook{BGG2022,
	author = {Brandolini, L. and Gariboldi, B. and Gigante, G.},
	booktitle = {The Mathematical Heritage of Guido Weiss},
	doi = {10.1007/978-3-031-76793-7_5},
	editor = {Hern{\'a}ndez, Eugenio and Peloso, Marco Maria and Ricci, Fulvio and Soria, Fernando and Tabacco, Anita},
	isbn = {978-3-031-76793-7},
	pages = {101--125},
	publisher = {The Mathematical Heritage of Guido Weiss, Springer Nature Switzerland},
	title = {Irregularities of distribution on two point homogeneous spaces},
	year = {2025}}

@article{BGGM2024,
	author = {Brandolini, L. and Gariboldi, B. and Gigante, G. and Monguzzi, A.},
	fjournal = {Mathematische Zeitschrift},
	issn = {0025-5874,1432-1823},
	journal = {Math. Z.},
	mrclass = {11K38 (33C45 43A85)},
	number = {3},
	pages = {Paper No. 52, 16},
	title = {Single radius spherical cap discrepancy on compact two-point homogeneous spaces},
	volume = {308},
	year = {2024}}

@article {BCT2023,
    AUTHOR = {Brandolini, L. and Colzani, L. and Travaglini,
              G.},
     TITLE = {Irregularities of distribution for bounded sets and
              half-spaces},
   JOURNAL = {Mathematika},
  FJOURNAL = {Mathematika. A Journal of Pure and Applied Mathematics},
    VOLUME = {69},
      YEAR = {2023},
    NUMBER = {1},
     PAGES = {68--89},
      ISSN = {0025-5793,2041-7942},
   MRCLASS = {11K38 (42B10)},
  MRNUMBER = {4516801},
MRREVIEWER = {Stefan\ Steinerberger},
       DOI = {10.1112/mtk.12178},
       URL = {https://doi.org/10.1112/mtk.12178},
}

@article {BCCGT2019,
    AUTHOR = {Brandolini, L. and Chen, W. W. L. and Colzani, L.
              and Gigante, G. and Travaglini, G.},
     TITLE = {Discrepancy and numerical integration on metric measure
              spaces},
   JOURNAL = {J. Geom. Anal.},
  FJOURNAL = {Journal of Geometric Analysis},
    VOLUME = {29},
      YEAR = {2019},
    NUMBER = {1},
     PAGES = {328--369},
      ISSN = {1050-6926,1559-002X},
   MRCLASS = {65D30 (11K38 30L05 31E05)},
  MRNUMBER = {3897016},
       DOI = {10.1007/s12220-018-9993-6},
       URL = {https://doi.org/10.1007/s12220-018-9993-6},
}

@article {GL2017,
    AUTHOR = {Gigante, G. and Leopardi, P.},
     TITLE = {Diameter bounded equal measure partitions of {A}hlfors regular
              metric measure spaces},
   JOURNAL = {Discrete Comput. Geom.},
  FJOURNAL = {Discrete \& Computational Geometry. An International Journal
              of Mathematics and Computer Science},
    VOLUME = {57},
      YEAR = {2017},
    NUMBER = {2},
     PAGES = {419--430},
      ISSN = {0179-5376,1432-0444},
   MRCLASS = {30L05 (11K38 28A75 52C22 54E45 65D30)},
  MRNUMBER = {3602860},
MRREVIEWER = {Juha\ Lehrb\"ack},
       DOI = {10.1007/s00454-016-9834-y},
       URL = {https://doi.org/10.1007/s00454-016-9834-y},
}

@article {GNT,
    AUTHOR = {Garg, R. and Nevo, A. and Taylor, K.},
     TITLE = {The lattice point counting problem on the {H}eisenberg groups},
   JOURNAL = {Ann. Inst. Fourier (Grenoble)},
  FJOURNAL = {Universit\'e{} de Grenoble. Annales de l'Institut Fourier},
    VOLUME = {65},
      YEAR = {2015},
    NUMBER = {5},
     PAGES = {2199--2233},
      ISSN = {0373-0956,1777-5310},
   MRCLASS = {11P21 (22E25 26D10 42B99 43A80)},
}

@article {G,
    AUTHOR = {Gath, Y. A.},
     TITLE = {On an analogue of the {G}auss circle problem for the
              {H}eisenberg groups},
   JOURNAL = {Ann. Sc. Norm. Super. Pisa Cl. Sci. (5)},
  FJOURNAL = {Annali della Scuola Normale Superiore di Pisa. Classe di
              Scienze. Serie V},
    VOLUME = {23},
      YEAR = {2022},
    NUMBER = {2},
     PAGES = {645--717},
      ISSN = {0391-173X,2036-2145},
 }

@book{F,
    AUTHOR = {Folland, G. B.},
     TITLE = {Harmonic analysis in phase space},
    SERIES = {Annals of Mathematics Studies},
    VOLUME = {122},
 PUBLISHER = {Princeton University Press, Princeton, NJ},
      YEAR = {1989},
     PAGES = {x+277},
      ISBN = {0-691-08527-7; 0-691-08528-5},
}

@book {T_book,
    AUTHOR = {Thangavelu, S.},
     TITLE = {Harmonic analysis on the {H}eisenberg group},
    SERIES = {Progress in Mathematics},
    VOLUME = {159},
 PUBLISHER = {Birkh\"auser Boston, Inc., Boston, MA},
      YEAR = {1998},
     PAGES = {xiv+192},
      ISBN = {0-8176-4050-9},
 }

@book {BB2023,
    AUTHOR = {Bramanti, Marco and Brandolini, Luca},
     TITLE = {H\"ormander operators},
 PUBLISHER = {World Scientific Publishing Co. Pte. Ltd., Hackensack, NJ},
      YEAR = {2023},
     PAGES = {xxviii+693},
      ISBN = {978-981-126-168-8; 978-981-126-169-5; 978-981-126-170-1},
   MRCLASS = {35Hxx},
  MRNUMBER = {4544986},
}

@book {BLU,
    AUTHOR = {Bonfiglioli, A. and Lanconelli, E. and Uguzzoni, F.},
     TITLE = {Stratified {L}ie groups and potential theory for their
              sub-{L}aplacians},
    SERIES = {Springer Monographs in Mathematics},
 PUBLISHER = {Springer, Berlin},
      YEAR = {2007},
     PAGES = {xxvi+800},
      ISBN = {978-3-540-71896-3; 3-540-71896-6},
   MRCLASS = {22E30 (31C45 35-02 35H10 43A80)},
  MRNUMBER = {2363343},
MRREVIEWER = {Maria\ Stella\ Fanciullo},
}

@article {Geller,
    AUTHOR = {Geller, Daryl},
     TITLE = {Fourier analysis on the {H}eisenberg group},
   JOURNAL = {Proc. Nat. Acad. Sci. U.S.A.},
  FJOURNAL = {Proceedings of the National Academy of Sciences of the United
              States of America},
    VOLUME = {74},
      YEAR = {1977},
    NUMBER = {4},
     PAGES = {1328--1331},
      ISSN = {0027-8424},
   MRCLASS = {22E30 (22E25 32F99 35H05 43A30)},
  MRNUMBER = {486312},
MRREVIEWER = {A.\ S.\ Dynin},
       DOI = {10.1073/pnas.74.4.1328},
       URL = {https://doi.org/10.1073/pnas.74.4.1328},
}

@article {BT2022,
    AUTHOR = {Brandolini, Luca and Travaglini, Giancarlo},
     TITLE = {Irregularities of distribution and geometry of planar convex
              sets},
   JOURNAL = {Adv. Math.},
  FJOURNAL = {Advances in Mathematics},
    VOLUME = {396},
      YEAR = {2022},
     PAGES = {Paper No. 108162, 40},
      ISSN = {0001-8708,1090-2082},
   MRCLASS = {11K38 (42B10)},
  MRNUMBER = {4358540},
MRREVIEWER = {Christoph\ Aistleitner},
       DOI = {10.1016/j.aim.2021.108162},
       URL = {https://doi.org/10.1016/j.aim.2021.108162},
}

@article {BBM2026,
    AUTHOR = {Brandolini, Luca and Travaglini, Giancarlo},
     TITLE = {Discrepancy estimates for non-smooth convex bodies: a case study},
     NOTE = {In preparation}
}

\end{document}